\documentclass{amsart}
\usepackage{ifthen}
\usepackage[normalem]{ulem}
\usepackage{graphicx}
\usepackage{amssymb}
\usepackage{psfrag}

\newtheorem{proposition}{Proposition}[section]
\newtheorem{theorem}[proposition]{Theorem}

\newtheorem{lemma}[proposition]{Lemma}
\newtheorem{prop}[proposition]{Proposition}
\newtheorem{cor}[proposition]{Corollary}

\theoremstyle{definition}

\theoremstyle{remark}
\newtheorem{remark}[proposition]{Remark}

\newcommand{\reals}{\mathbb R}

\newcommand{\set}[1]{{\left\lbrace #1 \right\rbrace}}
\newcommand{\pidown}{\pi_\downarrow}
\newcommand{\piup}{\pi^\uparrow}

\newcommand{\join}{\vee}
\newcommand{\meet}{\wedge}

\newcommand{\K}{\mathbb{K}}
\renewcommand{\H}{\mathcal{H}}
\newcommand{\std}{\operatorname{std}}
\newcommand{\inv}{\operatorname{inv}}
\newcommand{\rp}{\operatorname{rp}}
\newcommand{\rv}{\operatorname{rv}}
\newcommand{\PBT}{\operatorname{PBT}}
\newcommand{\dRec}{{\operatorname{dRec}}}
\newcommand{\dR}{{\operatorname{dR}}}
\newcommand{\tB}{{\operatorname{tB}}}
\newcommand{\tBax}{{\operatorname{tBax}}}
\newcommand{\Bax}{{\operatorname{Bax}}}
\newcommand{\tw}{{\operatorname{tw}}}
\newcommand{\TL}{{\operatorname{TL}}}
\newcommand{\BR}{{\operatorname{BR}}}
\newcommand{\A}{{\operatorname{A}}}
\newcommand{\Arr}{{\mathcal{A}}}
\newcommand{\F}{{\mathcal{F}}}
\newcommand{\B}{{\operatorname{B}}}
\renewcommand{\th}{^{\mbox{\footnotesize th}}}

\newcommand{\covered}{\lessdot}

\newcommand{\into}{\hookrightarrow}
\newcommand{\onto}{\twoheadrightarrow}
\newcommand{\som}[1]{\mbox{\sout{$#1$}}}

\usepackage[usenames]{color}

\newcommand{\margincolor}{Red}

\newcounter{margincounter}
\newcommand{\marginnum}{\textcolor{\margincolor}{\begin{picture}(0,0)\put(5,3){\circle{13}}\end{picture}\arabic{margincounter}}}

\newcommand{\margin}[1]
{\marginnum\marginpar{\textcolor{\margincolor}{\arabic{margincounter}} \tiny #1}\addtocounter{margincounter}{1}}

\newcommand{\proofread}
{
\ifthenelse{\isundefined{\margin}}
{
\special{!userdict begin /bop-hook{1.2 1.2 scale -51 -60 translate}def end}
}
{
\special{!userdict begin /bop-hook{1.2 1.2 scale -91 -60 translate}def end}
\setboolean{@mparswitch}{false} 
\addtolength{\marginparwidth}{-3mm}
}
}

\author{Shirley Law}
\author{Nathan Reading}
\address{Department of Mathematics, North Carolina State University, Raleigh, NC, USA}


\thanks{Nathan Reading was partially supported by NSA grant H98230-09-1-0056.}

\title{The Hopf algebra of diagonal rectangulations}

\begin{document}

\begin{abstract}
We define and study a combinatorial Hopf algebra dRec with basis elements indexed by diagonal rectangulations of a square. This Hopf algebra provides an intrinsic combinatorial realization of the Hopf algebra tBax of twisted Baxter permutations, which previously had only been described extrinsically as a Hopf subalgebra of the Malvenuto-Reutenauer Hopf algebra of permutations.
We describe the natural lattice structure on diagonal rectangulations, analogous to the Tamari lattice on triangulations, and observe that diagonal rectangulations index the vertices of a polytope analogous to the associahedron. We give an explicit bijection between twisted Baxter permutations and the better-known Baxter permutations, and describe the resulting Hopf algebra structure on Baxter permutations.
\end{abstract}


\maketitle

\vspace{-.3 in}

\setcounter{tocdepth}{2}

\section*{Contents}
\contentsline {section}{\tocsection {}{1}{Introduction}}{\pageref{intro}}
\contentsline {section}{\tocsection {}{2}{The Hopf algebra MR of permutations}}{\pageref{MR sec}}
\contentsline {subsection}{\tocsubsection {\quad\qquad}{2.1}{Product and coproduct}}{\pageref{MR subsec}}
\contentsline {subsection}{\tocsubsection {\quad\qquad}{2.2}{Automorphisms and antiautomorphisms}}{\pageref{auto sec}}
\contentsline {subsection}{\tocsubsection {\quad\qquad}{2.3}{Self-duality}}{\pageref{self sec}}
\contentsline {section}{\tocsection {}{3}{Some Hopf subalgebras and quotients of MR}}{\pageref{sub Hopf sec}}
\contentsline {subsection}{\tocsubsection {\quad\qquad}{3.1}{The Hopf algebras LR and LR$'$}}{\pageref{lr sec}}
\contentsline {subsection}{\tocsubsection {\quad\qquad}{3.2}{A pattern-avoidance description}}{\pageref{pattern sec}}
\contentsline {subsection}{\tocsubsection {\quad\qquad}{3.3}{The Hopf algebra NSym}}{\pageref{NSym sec}}
\contentsline {subsection}{\tocsubsection {\quad\qquad}{3.4}{Some quotients of MR}}{\pageref{quot sec}}
\contentsline {section}{\tocsection {}{4}{The Hopf algebra tBax of twisted Baxter permutations}}{\pageref{tBax Hopf sec}}
\contentsline {subsection}{\tocsubsection {\quad\qquad}{4.1}{$\mathcal {H}$-Families}}{\pageref{H fam sec}}
\contentsline {subsection}{\tocsubsection {\quad\qquad}{4.2}{Twisted Baxter permutations}}{\pageref{tBax sec}}
\contentsline {subsection}{\tocsubsection {\quad\qquad}{4.3}{The Hopf algebra tBax}}{\pageref{tBax Hopf subsec}}
\contentsline {section}{\tocsection {}{5}{The Hopf algebra dRec of diagonal rectangulations}}{\pageref{rec sec}}
\contentsline {subsection}{\tocsubsection {\quad\qquad}{5.1}{Diagonal rectangulations}}{\pageref{diag rec sec}}
\contentsline {subsection}{\tocsubsection {\quad\qquad}{5.2}{Product and coproduct}}{\pageref{Rec operations sec}}
\contentsline {section}{\tocsection {}{6}{The isomorphism from tBax to dRec}}{\pageref{isom sec}}
\contentsline {subsection}{\tocsubsection {\quad\qquad}{6.1}{Bijection}}{\pageref{bij sec}}
\contentsline {subsection}{\tocsubsection {\quad\qquad}{6.2}{Isomorphism}}{\pageref{isom subsec}}
\contentsline {section}{\tocsection {}{7}{The lattice and polytope of diagonal rectangulations}}{\pageref{lattice rec sec}}
\contentsline {section}{\tocsection {}{8}{Baxter permutations and twisted Baxter permutations}}{\pageref{Bax tBax sec}}
\contentsline {subsection}{\tocsubsection {\quad\qquad}{8.1}{Bijection}}{\pageref{Bax tBax bij}}
\contentsline {subsection}{\tocsubsection {\quad\qquad}{8.2}{Lattice and Hopf structures on Baxter permutations}}{\pageref{Bax lat Hopf sec}}
\contentsline {section}{\tocsection {}{}{Acknowledgments}}{\pageref{ack}}
\contentsline {section}{\tocsection {}{}{References}}{\pageref{bib}}

\section{Introduction}\label{intro}
Baxter permutations arose through the work of Baxter~\cite{Baxter} on pairs of commuting functions $f,g:[0,1]\to[0,1]$.
Interest in Baxter permutations increased when a pleasant formula was obtained \cite{CGHK} for the \emph{Baxter number} $B(n)$, the number of Baxter permutations in $S_n$:
\[B(n)=\binom{n+1}{1}^{-1}\binom{n+1}{2}^{-1}\,\,\sum_{k=1}^n\binom{n+1}{k-1}\binom{n+1}{k}\binom{n+1}{k+1}.\]
Baxter permutations also became an important example of generalized notions of pattern avoidance, for example in \cite{Bousquet-Melou,WestGTFS}.

The Baxter numbers made a surprising appearance in an algebraic context in~\cite{con_app}.
There, a lattice-theoretic construction produced an infinite family of Hopf subalgebras of the Malvenuto-Reutenauer \cite{Malvenuto,MR} Hopf algebra MR of permutations.
A prominent member of this infinite family is the Hopf algebra tBax, with basis elements indexed by \emph{twisted Baxter permutations}, which were computed in~\cite{con_app} to be counted by $B(n)$ for $n\le 15$.
West~\cite{West pers} later proved, using generating trees, that twisted Baxter permutations are counted by $B(n)$ for all $n$.
Other prominent members of the infinite family include the Hopf algebra NSym of noncommutative symmetric functions~\cite{GKLLRT} and the Loday-Ronco Hopf algebra LR of planar binary trees~\cite{LR}.
Just as NSym is the intersection of two anti-isomorphic copies of LR inside MR, the Hopf algebra tBax is the smallest member of the infinite family containing the two anti-isomorphic copies of LR.

Each of NSym and LR has a basis indexed by fundamental combinatorial objects (subsets and planar binary trees respectively), and the Hopf algebra operations (product and coproduct) are described intrinsically in terms of these objects.
By contrast, the operations in the Hopf algebra tBax are described by the construction in~\cite{con_app} only extrinsically.
That is, the operations are described only with reference to the embedding of tBax as a Hopf subalgebra of MR.
The starting point of this paper is an effort to provide tBax with an intrinsic description in terms of objects counted by $B(n)$.
A natural choice of objects is a certain class of \emph{rectangulations} (decompositions of a square into rectangles) shown by Ackerman, Barequet, and Pinter \cite[Lemma~5.2]{ABP} to be counted by $B(n)$.
These rectangulations, which we call \emph{diagonal rectangulations}, are closely related to an earlier construction by Dulucq and Guibert~\cite{DGStack} of certain ``twin'' pairs of planar binary trees.
They are also closely related to the \emph{mosaic floorplans}~\cite{HHCGDCG,YCCG} in the literature on VLSI circuit design.
(See Remark~\ref{mosaic}.)

The fundamental tool in this paper is a surjective map $\rho$ from permutations to diagonal rectangulations.
This map plays a role analogous to the map from permutations to trees which is central to the construction of LR.
In fact, $\rho$ is combinatorially very closely related to the map to trees.
We show that the fibers of $\rho$ are the congruence classes of a lattice congruence on the weak order, and show that the congruence in question coincides with the congruence used in~\cite{con_app} to define tBax as a Hopf subalgebra of MR.
From there, we conclude that $\rho$ restricts to a bijection between twisted Baxter permutations and diagonal rectangulations.
In particular, tBax can be realized as a Hopf algebra dRec with basis elements indexed by diagonal rectangulations.
The product and coproduct in dRec are shown to have natural descriptions in terms of the combinatorics of diagonal rectangulations.

The twisted Baxter permutations come with a natural lattice structure:
They are ordered as an induced subposet of the weak order, and this subposet is isomorphic to the quotient of the weak order modulo the lattice congruence mentioned above.
This lattice can be interpreted as a lattice on diagonal rectangulations, analogous to the well-known Tamari lattice~\cite{Tamari} on triangulations of a convex polygon.
We characterize this lattice in terms of certain local moves which we call \emph{pivots}, analogous to diagonal flips on triangulations.
The Hasse diagram of the lattice of diagonal rectangulations, thought of as an undirected graph, is analogous to the $1$-skeleton of the associahedron.  
This graph, which was also considered in~\cite{ABP}, is the $1$-skeleton of a convex polytope analogous to the associahedron.
Specifically, taking the Minkowski sum of two realizations of the associahedron, we obtain a polytope whose vertices are indexed by diagonal rectangulations and whose edges correspond to pivots.

The map $\rho$ is also instrumental in providing an explicit (but recursive) bijection between Baxter permutations and twisted Baxter permutations.
The bijection is the restriction of the downward projection map associated to a certain lattice congruence on the weak order, and its inverse is the restriction of the corresponding upward projection map.
Using the bijection, we construct a Hopf algebra Bax indexed by Baxter permutations, isomorphic to tBax and dRec, and having a simple description analogous to the original definition of tBax.
Also as a consequence of the bijection, we show that the lattice of diagonal rectangulations is isomorphic not only to the subposet of the weak order induced by twisted Baxter permutations, but also to the subposet of the weak order induced by Baxter permutations.
\emph{A priori} it is not clear that the latter is even a lattice.

The constructions presented here are informed by previous research on objects counted by the Baxter number, including in particular \cite{ABP,DGStack,FFNO}, but our proofs do not depend on any of these previous results.
One part of our argument (the proof of Lemma~\ref{Rec to Bax}) is, however, borrowed from \cite{FFNO}.
For the most part, our proofs fall naturally out of the lattice-theoretic approach to Hopf algebras found in~\cite{con_app}, and out of  detailed combinatorial/polytopal investigations, found in~\cite{cambrian}, of maps from permutations to triangulations.

In a paper~\cite{Giraudo} that appeared on the arXiv just after the present paper, Giraudo rediscovers the Hopf algebra tBax from~\cite{con_app} and studies tBax via twin binary trees (see Remark~\ref{credit remark}) and Baxter permutations.
Among the results of~\cite{Giraudo} that are not in the present paper are a multiplicative basis for Bax, freeness, cofreeness, self-duality, a bidendriform structure, and primitive elements.
Among the results of the present paper that are not in~\cite{Giraudo} are the intrinsic description of the coproduct on dRec, the intrinsic description of the product on dRec in terms of the defining basis, a complete description of cover relations in the lattice of diagonal rectangulations, the polytope construction, and the direct bijection between Baxter permutations and twisted Baxter permutations.

\section{The Hopf algebra MR of permutations}\label{MR sec}

In this section, we describe the Malvenuto-Reutenauer Hopf algebra MR, quoting results of~\cite{Malvenuto,MR}.
For more details on Hopf algebras in general, see \cite{MM,Sweedler}.

\subsection{Product and coproduct}\label{MR subsec}
Let $\K$ be a field.
For each integer $n\ge 0$, let $\K[S_n]$ be an $(n!)$-dimensional vector space over $\K$ with basis indexed by the elements of the symmetric group $S_n$.
For notational convenience, we replace each basis element with the permutation indexing it.  
Thus the elements of $\K[S_n]$ are represented as $\K$-linear combinations of permutations in $S_n$.
The Malvenuto-Reutenauer Hopf algebra of permutations is the graded vector space $\K[S_\infty]=\bigoplus_{n\ge 0}\K[S_n]$, endowed with a product and coproduct that we now define.

Let $p$,~$q$ and $n$ be nonnegative integers with $p+q=n$.
Let $x=x_1x_2\cdots x_p\in S_p$ and $y=y_1y_2\cdots y_q\in S_q$.
A \emph{shifted shuffle} of $x$ and $y$ is a permutation $z=z_1\cdots z_n\in S_n$ such that
\begin{enumerate}
\item[(i)]  the subsequence of $z$ consisting of the values $1,2,\ldots,p$ equals the sequence $x_1x_2\cdots x_p$, and 
\item[(ii) ] the subsequence of $z$ consisting of the values $p+1,p+2,\ldots,n$ equals the sequence $y_1+p,y_2+p,\ldots,y_q+p$.
\end{enumerate}
Said another way, to construct a shifted shuffle of $x$ and $y$, one first ``shifts'' the entries of $y$ up by $p$, and then ``shuffles'' the sequences $x_1,x_2,\ldots,x_p$ and $y_1+p,y_2+p,\ldots,y_q+p$ together, preserving the relative order of elements within each sequence.
Because of the shifting, this process is not symmetric in $x$ and $y$.
There are $\binom{n}{p}$ shifted shuffles of $x$ and $y$.

The product $x\bullet_Sy$ is the sum of all of the shifted shuffles of $x$ and $y$.
(That is, a linear combination with all coefficients equal to one.)
Thus, for example, 
\begin{align*}
21\bullet_S 21&=2143+2413 +2431+4213+4231+4321.\\
21\bullet_S 132&=21354 + 23154+23514+23541\\& \qquad+ \,32154 +32514+32541+35214+35241+35421.
\end{align*}
We can think of the product as a map of graded vector spaces, from $\K[S_\infty]\otimes\K[S_\infty]$ to $\K[S_\infty]$, taking a basis element $x\otimes y$ to $x\bullet_S y$.
This is a \emph{graded product} on $\K[S_\infty]$, in the sense that, for each $p$ and~$q$, it restricts to a map from $\K[S_p]\otimes\K[S_q]$ to $\K[S_{p+q}]$.

The \emph{standardization} of a sequence $a_1,a_2,\ldots,a_k$ of distinct integers is the unique permutation $x\in S_k$ such that each pair $i\neq j$ in $[k]$ has $x_i<x_j$ if and only if $a_i<a_j$.
The notation $\std(a_1,a_2,\ldots,a_k)$ stands for the standardization of $a_1,a_2,\ldots,a_k$.
For example, $\std(7,3,2,11,5)=42153$.

The coproduct $\Delta_S$ in $\K[S_\infty]$ maps $x=x_1x_2\cdots x_n\in S_n$ to 
\[\Delta_S(x)=\sum_{i=0}^n\std(x_1,\ldots,x_i)\otimes\std(x_{i+1},\ldots,x_n).\]
The notations $x_1,\ldots,x_0$ and $x_{n+1},\ldots,x_n$ both stand for the empty sequence, so $std(x_1,\ldots,x_0)$ and $\std(x_{n+1},\ldots,x_n)$ both denote the empty permutation $\emptyset\in S_0$.
Thus, for example, 
\begin{align*}
\Delta_S(563241)&=\emptyset\otimes563241+1\otimes53241+12\otimes3241\\
	&\qquad+\,231\otimes231+3421\otimes21+45213\otimes1+563241\otimes\emptyset.
\end{align*}
This map of graded vector spaces from $\K[S_\infty]$ to $\K[S_\infty]\otimes\K[S_\infty]$ is a \emph{graded coproduct}, because it restricts to a map from $\K[S_n]$ to $\sum_{p+q=n}\K[S_p]\otimes\K[S_q]$.

The graded vector space $\K[S_\infty]$, endowed with the product $\bullet_S$ and coproduct $\Delta_S$, are a \emph{graded Hopf algebra} $(\K[S_\infty],\bullet_S,\Delta_S)$, meaning that the product satisfies the axioms of a graded algebra, the coproduct satisfies the axioms of a graded coalgebra (which are dual to the axioms of an algebra), and the product and coproduct satisfy a certain compatibility condition.  
This is the Malvenuto-Reutenauer Hopf algebra of permutations, which we refer to as MR.

\subsection{Automorphisms and antiautomorphisms}\label{auto sec}
An important theme running through the study of MR is the weak order.
The weak order on permutations is the partial order whose cover relations are described as follows:  Given $x=x_1\cdots x_n\in S_n$, one moves up by a cover by swapping two  adjacent entries $x_i$ and $x_{i+1}$ with $x_i<x_{i+1}$.
(That is, the higher permutation in the covering pair agrees with $x$, except that the values $x_i$ and $x_{i+1}$, which occur in $x$ in numerical order, are swapped so as to occur in the higher permutation out of numerical order.)
The weak order on $S_4$ is illustrated in Figure~\ref{S4 weak fig}.

\begin{figure}[t]
\centerline{\scalebox{.4}{\includegraphics{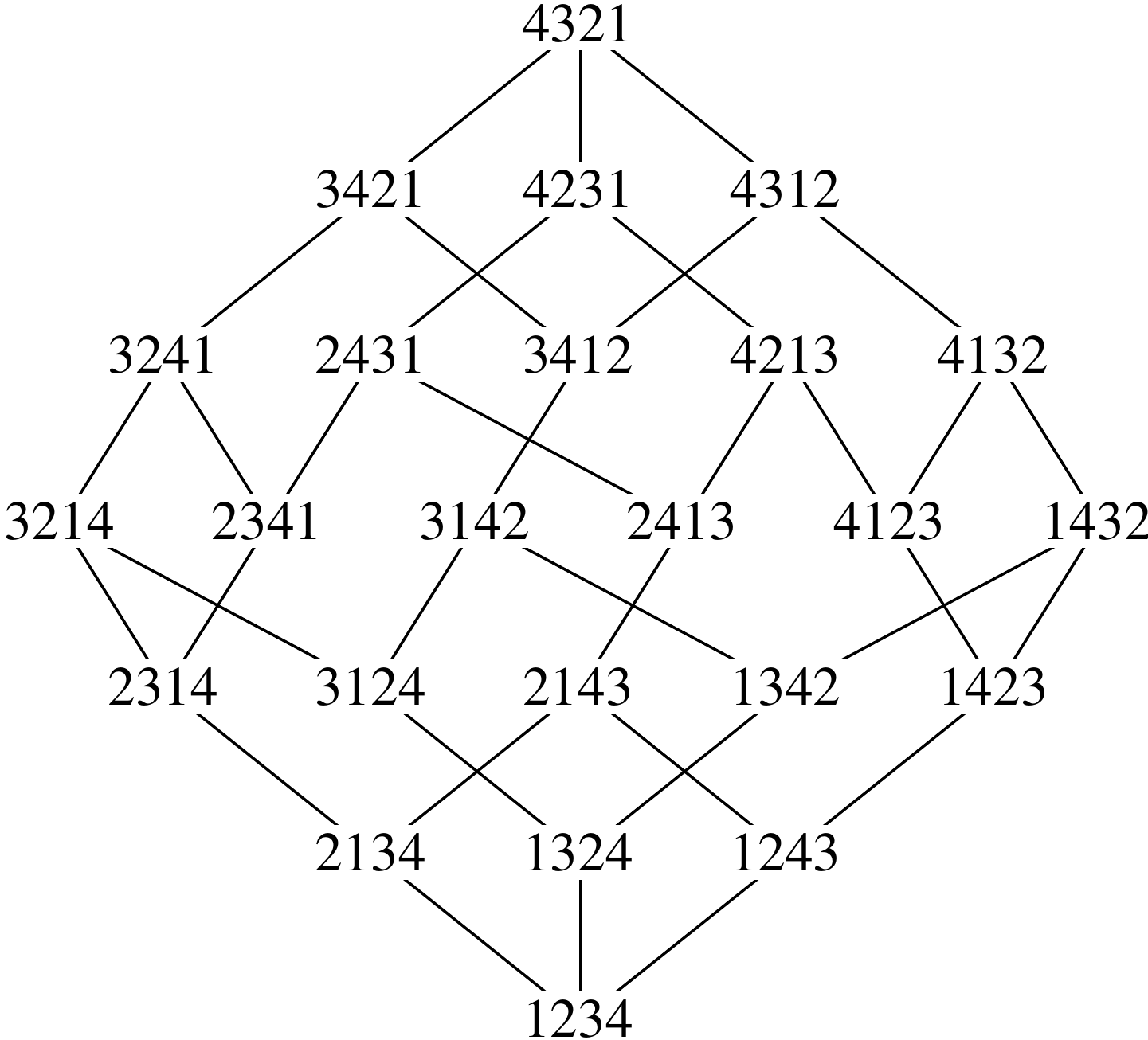}}}
\caption{The weak order on $S_4$}
\label{S4 weak fig}
\end{figure}

The importance of the weak order to the Malvenuto-Reutenauer Hopf algebra begins with the observation \cite[Theorem~4.1]{LRorder} that the product $x\bullet_S y$ is the sum over all permutations in a certain interval in the weak order on $S_{p+q}$.
The connection to the weak order also leads us to consider how automorphisms and antiautomorphisms of the weak order interact with the product and coproduct.

Borrowing from the terminology of Coxeter groups, for each $n$, let $w_0$ be the permutation $n(n-1)\cdots 21$, which is the unique maximal element of the weak order.
The map $x\mapsto w_0\cdot x$ is an antiautomorphism of the weak order.
(Here, ``$\cdot$'' is the usual composition product of permutations, and we follow the convention of composing permutations as functions, from right to left.)
We write $\rv$ for this map, because the effect of the map is to reverse the values of the one-line notation for $x$, by replacing each value $x_i$ by $n+1-x_i$.
For example, $\rv(625431)=654321\cdot 625431=152346$.
Extending the map $\rv$ by linearity to a map on $\K[S_\infty]$. we obtain an algebra antiautomorphism (\emph{i.e.}\ an antiautomorphism of the product) and a coalgebra automorphism (\emph{i.e.}\ an automorphism of the coproduct) of MR.
The following examples are representative:
\begin{align*}
\rv(12)\bullet_S\rv(312)&=21\bullet_S132\\
&=21354+23154+23514+23541\\&\qquad+\,32154+32514+32541+35214+35241+35421\\
&=\rv(45312+43512+43152+43125\\&\qquad+\,34512+34152+34125+31452+31425+31245)\\
&=\rv(312\bullet_S12).
\end{align*}
\begin{align*}
\Delta_S(\rv(214536))&=\Delta_S(563241)\\
&=\emptyset\otimes563241+1\otimes53241+12\otimes3241\\
	&\qquad+\,231\otimes231+3421\otimes21+45213\otimes1+563241\otimes\emptyset\\
&=(\rv\otimes\rv)(\emptyset\otimes214536+1\otimes13425+21\otimes2314\\
	&\qquad+\,213\otimes213+2134\otimes12+21453\otimes1+214536\otimes\emptyset)\\
&=(\rv\otimes\rv)(\Delta_S(214536)).
\end{align*}

The map $x\mapsto x\cdot w_0$ is also an antiautomorphism of the weak order.
We write $\rp$ for this map, because the map reverses positions of the entries of the one-line notation for $x$.
For example, $\rp(624531)=624531\cdot 654321=135426$.
Extending by linearity, we obtain an algebra automorphism and a coalgebra antiautomorphism of MR, as illustrated by the following representative examples:
\begin{align*}
\rp(12)\bullet_S\rp(231)&=21\bullet_S132\\
&=21354+23154+23514+23541\\&\qquad+\,32154+32514+32541+35214+35241+35421\\
&=\rp(45312+45132+41532+14532\\&\qquad+\,45123+41523+14523+41253+14253+12453)\\
&=\rp(12\bullet_S231).
\end{align*}
\begin{align*}
\Delta_S(\rp(142365))&=\Delta_S(563241)\\
&=\emptyset\otimes563241+1\otimes53241+12\otimes3241\\
	&\qquad+\,231\otimes231+3421\otimes21+45213\otimes1+563241\otimes\emptyset\\
&=(\rp\otimes\rp)(\emptyset\otimes142365+1\otimes14235+21\otimes1423\\
	&\qquad+\,132\otimes132+1243\otimes12+31254\otimes1+142365\otimes\emptyset)\\
&=\tw\circ(\rp\otimes\rp)(\Delta_S(142365)),
\end{align*}
where $\tw$ is the linear map that takes $a\otimes b$ to $b\otimes a$.
The maps $\rv$ and $\rp$ commute, and the composition $\rv\circ\rp$ is an antiautomorphism of MR (as a graded Hopf algebra, \emph{i.e.}\ both as an algebra and as a coalgebra).

\subsection{Self-duality}\label{self sec}
The weak order defined above is more specifically known as the \emph{right weak order}.
There is also a \emph{left weak order} on permutations.
The right and left weak orders are isomorphic via the involution $x\mapsto x^{-1}$.
For the purposes of this paper, we take this isomorphism as the definition of the left weak order.
We also use the term ``weak order'' as an abbreviation for ``right weak order.''
The Hopf algebra MR is \emph{self-dual}.
More specifically, the map $\inv:x\mapsto x^{-1}$ extends linearly to an isomorphism \cite{Malvenuto} between MR and its graded dual MR$^*$.
We now explain what this means.

The graded dual of a graded vector space $V=\bigoplus_{n\ge 0}V_n$ is the graded vector space $V^*=\bigoplus_{n\ge 0}V^*_n$, where each $V^*_n$ is the dual vector space to $V_n$.
If each $V_n$ is finite, then we can identify $V^*_n$ with $V_n$, and thus identify $V^*$ with $V$.
If $V$ furthermore has the structure of a graded Hopf algebra, then $V^*$ is also a graded Hopf algebra, the \emph{graded dual} of $V$.
We define the graded dual Hopf algebra in the context of $\K[S_\infty]$, using a permutation $x$ as a shorthand both for the basis element indexed by $x$ and for the dual basis element indexed by $x$.

Dual to the product $\bullet_S$ (a map from $\K[S_\infty]\otimes\K[S_\infty]$ to $\K[S_\infty]$) is a coproduct $\bullet_S^*$ (a map from $\K[S_\infty]$ to $\K[S_\infty]\otimes\K[S_\infty]$) defined as follows:
Given permutations $z\in S_n$, $x\in S_p$ and $y\in S_q$, the coefficient of $x\otimes y$ in $\bullet_S^*(z)$ equals the coefficient of $z$ in $x\bullet_S y$.
Similarly, there is a product $\Delta_S^*$ dual to the coproduct $\Delta_S$ such that the coefficient of $z$ in $x\Delta_S^*y$ equals the coefficient of $x\otimes y$ in $\Delta_S(z)$.

It is immediate that the map $\inv$ extends linearly to an automorphism of graded vector spaces of $\K[S_\infty]$.
To see that this is an isomorphism between the graded Hopf algebra MR$=(\K[S_\infty],\bullet_S,\Delta_S)$ and its graded dual MR$^*=(\K[S_\infty],\Delta_S^*,\bullet_S^*)$, we must check that $(\inv\otimes\inv)\circ\Delta_S=(\bullet_S^*)\circ\inv$ and that $\inv\circ(\bullet_S)=\Delta_S^*\circ(\inv\otimes\inv)$.
Both conditions amount to checking that the coefficient of $x\otimes y$ in $\Delta_S(z)$ equals the coefficient of $z^{-1}$ in $x^{-1}\bullet_Sy^{-1}$ for all $x$, $y$, and $z$.
This is shown, for example in \cite[Theorem~9.1]{AgSoMR}. 

\begin{remark}\label{ours and MR's}
The original definition of \cite{Malvenuto,MR} constructs what we are calling the dual Hopf algebra MR$^*$.
We prefer the definition given here because it allows us to work with the right weak order, rather than translating the results of \cite{congruence,con_app,cambrian} into results about the weak order.
The difference in the two definitions is minor, in the sense that one can switch definitions by replacing all permutations by their inverses.
\end{remark}

\section{Some Hopf subalgebras and quotients of MR}\label{sub Hopf sec}

Many important combinatorial Hopf algebras are Hopf subalgebras or quotients of MR.
One of the most important is the Hopf algebra of Noncommutative symmetric functions~\cite{GKLLRT}, which can be defined on a graded vector space with basis elements at degree $n$ indexed by subsets of $[n-1]$.
Another important Hopf subalgebra of MR is the Loday-Ronco Hopf algebra, defined on a graded vector space with basis elements indexed by planar binary trees.
This is isomorphic~\cite{AgSoLR,Foissy,Holtkamp,Palacios} to the noncommutative Connes-Kreimer Hopf algebra of~\cite[Section~5]{Foissy}.

In this section, we describe these Hopf subalgebras, with particular emphasis on the Loday-Ronco Hopf algebra.
Our description is fairly detailed, for two reasons:  First, the detailed description gives context to the discussion of the Hopf algebra tBax of twisted Baxter permutations/Hopf algebra dRec of diagonal rectangulations discussed in Sections~\ref{tBax Hopf sec} and~\ref{rec sec}.
Second, many constructions and definitions developed in this section are used directly, in the later sections, in analyzing tBax and dRec.  

None of the results in this section are new, but they are presented in a way particularly suited to our purposes. 
For different accounts of the same material, see, for example, \cite{AgSoLR,LR,LRorder}.

\subsection{The Hopf algebras LR and LR$'$}\label{lr sec}
A \emph{planar binary tree} is a rooted tree (a tree with a distinguished vertex called the~\emph{root}) in which every vertex either has no children or has two children, one designated as a \emph{right} child and one as a \emph{left child}.
Vertices with no children are called \emph{leaves}.
For the moment, we draw the trees with their leaves lined up horizontally and the root above them, with edges drawn so that, for each leaf, the path from that leaf to the root is monotone up.
We call planar binary trees drawn in this manner \emph{top trees}, for reasons that are apparent in Section~\ref{bij sec}.
Let $\PBT^t_n$ be the set of top planar binary trees with $n+1$ leaves.
For example, the five trees of $\PBT^t_3$ are shown in Figure~\ref{PBTt3}.

\begin{figure}[t]
\begin{tabular}{ccccc}
\includegraphics{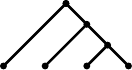}&
\includegraphics{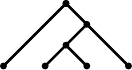}&
\includegraphics{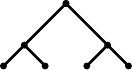}&
\includegraphics{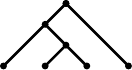}&
\includegraphics{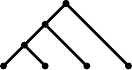}
\end{tabular}
\caption{The five top planar binary trees in $\PBT^t_3$.}
\label{PBTt3}
\end{figure}

To realize the Loday-Ronco Hopf algebra of planar binary trees as a Hopf subalgebra of MR, we define a map $\rho_t:S_n\to\PBT^t_n$.
Versions of this map have appeared in many papers, including \cite{Iterated,NonpureII,LR,LRorder,cambrian,Tonks}.
Let $x=x_1x_2\cdots x_n\in S_n$.
Write $n+1$ points on a horizontal line and label the spaces between points $1,2,\ldots,n$, from left to right.
The map $\rho_t$ constructs the tree starting at the leaves and uses each entry of $x$ to combine leaves into subtrees, as follows.
The two leaves adjacent to the label $x_1$ are first joined to a vertex between and above them.
The number $x_1-1$ is then thought of as labeling the space between the new vertex and the leaf to its left.
Similarly, $x_1+1$ labels the space between the new vertex and the leaf to its right.
The process continues until the last entry of $x$ is read, causing the addition of a final new vertex,  the root.
Figure~\ref{rho t fig} illustrates the steps in the construction of $\rho_t(x)$ when $x=467198352$.

\begin{figure}[t]
\begin{tabular}{ccc}
\scalebox{.95}{\includegraphics{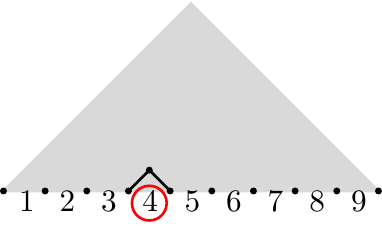}}&
\scalebox{.95}{\includegraphics{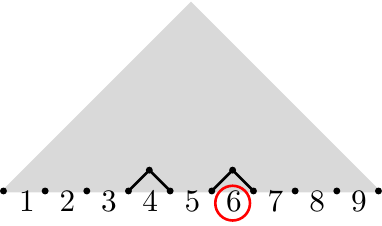}}&
\scalebox{.95}{\includegraphics{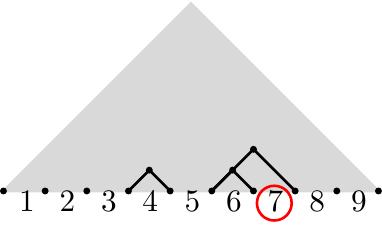}}\\[10 pt]
\scalebox{.95}{\includegraphics{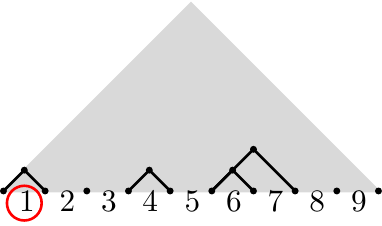}}&
\scalebox{.95}{\includegraphics{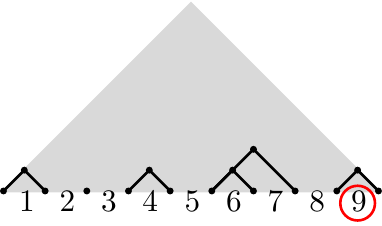}}&
\scalebox{.95}{\includegraphics{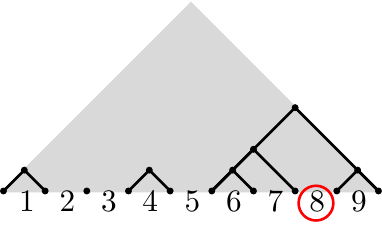}}\\[10 pt]
\scalebox{.95}{\includegraphics{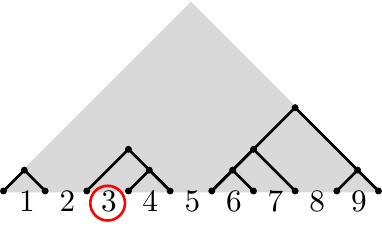}}&
\scalebox{.95}{\includegraphics{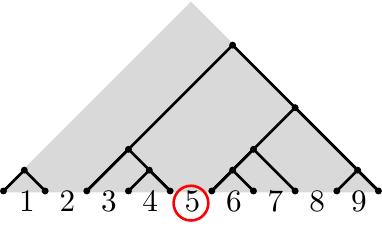}}&
\scalebox{.95}{\includegraphics{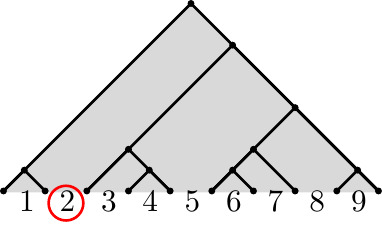}}
\end{tabular}
\caption{Steps in the construction of $\rho_t(467198352)$}
\label{rho t fig}
\end{figure}

For any (top) planar binary tree $T\in\PBT^t_n$, let $c(T)$ be the following sum in $\K[S_\infty]$:
\[c(T)=\sum_{\substack{x\in S_n\\\rho_t(x)=T}}x.\]
The elements of the form $c(T)$ constitute a graded basis for a graded Hopf subalgebra of MR, called the Loday-Ronco Hopf algebra LR.
The Hopf algebra LR can be viewed as a Hopf algebra on the graded vector space $\K[\PBT^t_\infty]=\bigoplus_{n\ge 0}\K[\PBT^t_n]$.

\begin{remark}\label{ours and LR's}
In~\cite{LR,LRorder}, the dual definition of MR is taken.
(See Remark~\ref{ours and MR's}.)
The construction of LR in \cite{LR} proceeds via a map $\psi$ which sends a permutation $x$ to the reflection, through a horizontal line, of the tree $\rho_t(\inv(x))$.
\end{remark}

There is an anti-isomorphic copy LR$'$ of LR that can be realized in MR in a similar way.
To construct this anti-isomorphic copy, we now draw \emph{bottom} planar binary trees.
That is, we draw the trees with the leaves lined up horizontally and the root below them, with edges drawn so that, for each leaf, the path from that leaf to the root is monotone down.
The notation $\PBT^b_n$ stands for the set of bottom planar binary trees with $n+1$ leaves.
Now LR$'$ is constructed in a similar manner to the construction of LR, using a different map $\rho_b$ from permutations to trees.
Once again, write $n+1$ points on a horizontal line and label the spaces between points $1,2,\ldots,n$ from left to right.
We may as well, right from the start, draw a straight line from the root to the leftmost leaf, and a straight line from the root to the rightmost leaf.
A planar binary tree has a uniquely defined \emph{left subtree} and \emph{right subtree}.
The tree $\rho_b(x)$ is constructed so that $x_1$ labels the unique space between the two subtrees.
Then $x_2$ labels the space between two \textbf{sub}-subtrees, \emph{etc}.
Step by step construction of $\rho_b(x)$ in the case $x=467198352\in S_9$ is illustrated in Figure~\ref{rho b fig}.

\begin{figure}[t]
\begin{tabular}{ccc}
\scalebox{.95}{\includegraphics{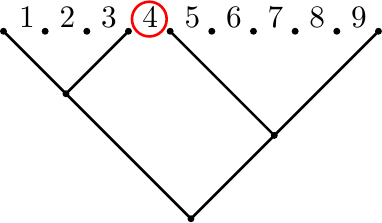}}&
\scalebox{.95}{\includegraphics{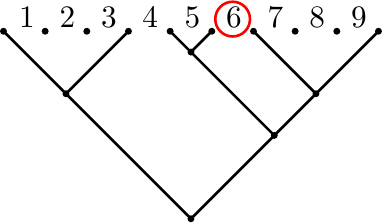}}&
\scalebox{.95}{\includegraphics{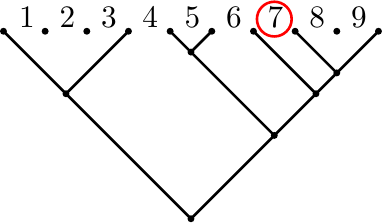}}\\[10 pt]
\scalebox{.95}{\includegraphics{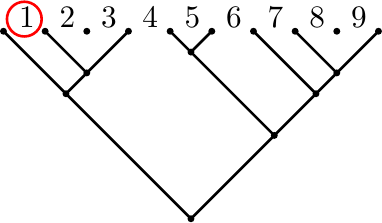}}&
\scalebox{.95}{\includegraphics{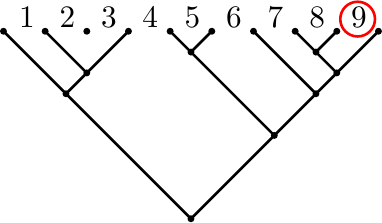}}&
\scalebox{.95}{\includegraphics{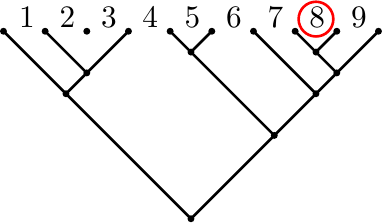}}\\[10 pt]
\scalebox{.95}{\includegraphics{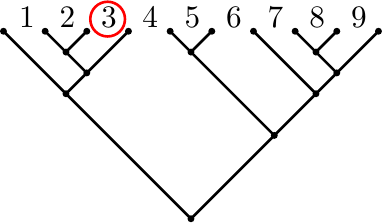}}&
\scalebox{.95}{\includegraphics{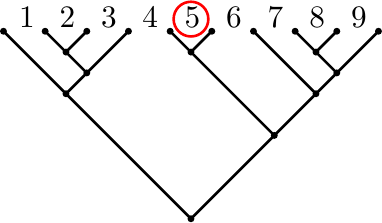}}&
\scalebox{.95}{\includegraphics{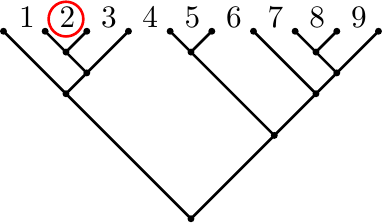}}
\end{tabular}
\caption{Steps in the construction of $\rho_b(467198352)$}
\label{rho b fig}
\end{figure}

Define for any bottom planar binary tree $T\in\PBT^b_n$ an element $c'(T)\in\K[S_\infty]$:
\[c'(T)=\sum_{\substack{x\in S_n\\\rho_b(x)=T}}x.\]
The elements of the form $c'(T)$ form the graded basis of a graded Hopf subalgebra LR$'$ of MR.
We view LR$'$ as a Hopf algebra on the graded vector space $\K[\PBT^b_\infty]=\bigoplus_{n\ge 0}\K[\PBT^b_\infty]$.
The anti-isomorphism between LR and LR$'$ is described in Section~\ref{pattern sec}.

\subsection{A pattern-avoidance description}\label{pattern sec}
In \cite[Section~4]{cambrian}, a map $\eta$ is described from permutations in $S_n$ to triangulations of a convex $(n+2)$-gon $Q$.
(When reading \cite[Section~4]{cambrian} solely for the definition of $\eta$, one may safely skip any paragraph containing the phrase ``fiber polytope'' and take the few remaining paragraphs as a description of $Q$ and $\eta$.)
The vertices of $Q$ are labeled $0$ through $n+1$, with certain conditions on the configuration of the vertices, and the map $\eta$ depends on the labeling/configuration.
We are first interested in a specific choice of $Q$ (and below, in another specific choice).
Namely, take the vertices $0$ and $n+1$ to form a horizontal edge of $Q$, with all of the remaining vertices above that edge, with labels increasing left to right.
The combinatorics of this choice is summarized, in the notation of \cite[Section~4]{cambrian}, by putting an upper-bar on each index in $[n]$.

Given a permutation $x\in S_n$, one readily verifies from the definition of $\eta$ that the dual planar binary tree to $\eta(x)$, when properly ``straightened out,'' is $\rho_b(x)$.
For an example (and as an illustration of the notion of a dual tree), compare Figure~\ref{eta rho b} to Figure~\ref{rho b fig}.
Since a triangulation of $Q$ is determined uniquely by its dual tree, we can rephrase (and specialize) results on $\eta$ from \cite[Section~5]{cambrian} to obtain a characterization of the fibers of $\rho_b$.

\begin{figure}[t]
\scalebox{1.1}{\includegraphics{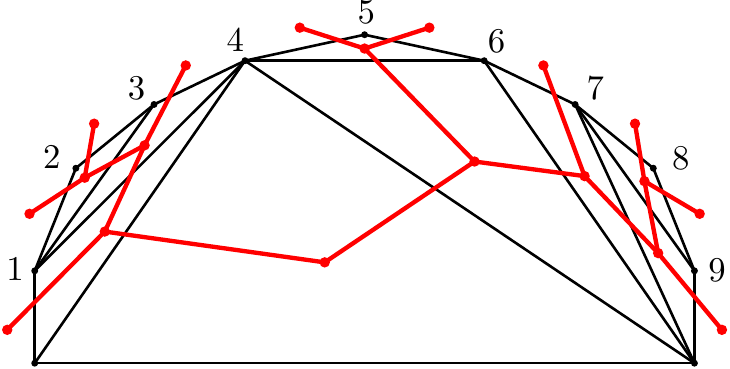}}
\caption{$\eta(467198352)$ and its dual tree $\rho_b(467198352)$}
\label{eta rho b}
\end{figure}

This characterization uses the fact that the weak order on $S_n$ is a lattice.
A congruence on a lattice is an equivalence relation respecting meet and join in the same sense that a congruence of rings respects addition and multiplication.
The proposition below follows from \cite[Theorem~5.1]{cambrian}.

\begin{prop}\label{rho b cong}
The fibers of $\rho_b$ are the congruence classes of a congruence on the weak order on $S_n$.
\end{prop}

By a general fact on congruences of finite lattices (see, for example, \cite[Section~2]{con_app}), Proposition~\ref{rho b cong} implies that each fiber of $\rho_b$ is an interval in the weak order.
In particular, each fiber has a unique minimal representative and a unique maximal representative.
These representatives can be described by \emph{pattern avoidance}.
For $1\le k \le n$, let $x=x_{1}x_{2}\cdots x_{k}$ be a permutation in $S_{k}$ and let $y=y_1y_2\cdots y_n$ be a permutation in $S_{n}$.
We say $y$ contains the pattern $x_1$--\,$x_2$--\,$\cdots$--\,$x_k$ if there exist indices $1\le i_{1}<i_{2}< \cdots < i_{k}\le n$ such that $y_{i_{j}}<y_{i_{l}}$ if and only if $x_{j}<x_{l}$ for all $j,l \in [k]$.  
Otherwise we say $y$ \emph{avoids} the pattern $x_1$--\,$x_2$--\,$\cdots$--\,$x_k$.  
Here, we are using the generalized pattern notation of~\cite{Gen Pat}.
In the more common notation, the dashes in $x_1$--\,$x_2$--\,$\cdots$--\,$x_k$ are omitted and one speaks of $y$ containing or avoiding $x_1x_2\cdots x_k$.
The reasons for this more involved notation are explained later, in Section~\ref{tBax sec}.
See \cite{Wilf} for more details on pattern avoidance.
The proposition below follows from \cite[Proposition~5.7]{cambrian}.

\begin{prop}\label{rho b avoid}

\noindent
\begin{enumerate}
\item A permutation is minimal in its $\rho_b$-fiber if and only if it avoids $2$-$3$-$1$.
\item A permutation is maximal in its $\rho_b$-fiber if and only if it avoids $2$-$1$-$3$.
\end{enumerate}
\end{prop}

A complete description of the fibers is as follows:
Suppose $x\covered y$ in the weak order.   
Then there exist a pair of entries $a$ and $c$ in $y$ with $a$ immediately following $c$ in $y$ and $a<c$, such that $x$ agrees with $y$ everywhere except that $a$ and $c$ are swapped.
We say that $x$ is obtained from $y$ by a \emph{$(231\to213)$-move} if the entries $a$ and $c$ are the ``$1$'' and ``$3$'' in some $231$-pattern, or in other words, if there exists an entry $b$ with $a<b<c$ such that $b$ occurs somewhere left of $ca$ in $y$.
We also say that $y$ is obtained from $x$ by a \emph{$(213\to231)$-move}.
The proposition below follows from \cite[Proposition~5.3]{cambrian}.

\begin{prop}\label{rho b move}
Let $x\covered y$ in the weak order.
Then $\rho_b(x)=\rho_b(y)$ if and only if $x$ is obtained from $y$ by a $(231\to213)$-move.
\end{prop}

Let $\Theta_{231}$ be the congruence whose classes are the fibers of $\rho_b$.
Let $\pidown^{231}$ stand for the map which sends an element $x\in S_n$ to the minimal element of its $\Theta_{231}$-class.
Let $\piup_{231}$ be the map sending $x$ to the maximal element of its $\Theta_{231}$-class.
Propositions~\ref{rho b avoid} and~\ref{rho b move}, along with the fact that $\Theta_{231}$-classes are intervals, imply the following recursive characterizations of the maps $\pidown^{231}$ and $\piup_{231}$.

\begin{prop}\label{pidown 231 char}
Let $y\in S_n$.
If some permutation $x\in S_n$ can be obtained from~$y$ by a $(231\to213)$-move, then $\pidown^{231}(y)=\pidown^{231}(x)$.
Otherwise, $\pidown^{231}(y)=y$. 
\end{prop}

\begin{prop}\label{piup 231 char}
Let $x\in S_n$.
If some permutation $y\in S_n$ can be obtained from~$x$ by a $(213\to231)$-move, then $\piup_{231}(x)=\piup_{231}(y)$.
Otherwise, $\piup_{231}(x)=x$. 
\end{prop}

The concatenation of Proposition~\ref{pidown 231 char} with the previous two propositions implies in particular that $y$ admits a $(231\to213)$-move if and only if it contains the pattern $2$-$3$-$1$.
This is easily checked directly.  (See, for example, \cite[Lemma~5.5]{cambrian}.)
A similar statement holds for the pattern $2$-$1$-$3$.

These considerations allow us to think of LR$'$ as the Hopf algebra of ($2$-$3$-$1$)-avoiding permutations, using the map $\pidown^{231}$ in place of $\rho_b$.
The quotient of the weak order modulo the congruence $\Theta_{231}$ is isomorphic \cite{NonpureII,cambrian} to the well-known \emph{Tamari lattice} first defined in~\cite{Tamari}.
By another general fact about finite lattices (found, for example, in \cite[Section~2]{con_app}), this quotient is isomorphic to the subposet induced by the set of minimal congruence class representatives.
In other words, it is the restriction of the weak order to ($2$-$3$-$1$)-avoiding permutations.
This restriction is also a sublattice of the weak order~\cite{NonpureII}.
The Tamari lattice plays a role for LR$'$ analogous to the role the weak order plays for MR.

All of this development can be adapted to the map $\rho_t$, which corresponds to a different instance of the map $\eta$.  
In this case, we define a polygon similarly to the polygon $Q$ described above, with the vertices $1,\ldots,n$ occurring \textbf{below} the horizontal edge $0$---$n+1$.
This corresponds, in the notation of \cite[Section~4]{cambrian}, to putting a lower-bar on each index in $[n]$.
Using results of \cite[Section~5]{cambrian} in a similar way, we obtain analogous results for $\rho_t$.

\begin{prop}\label{rho t cong}
The fibers of $\rho_t$ are the congruence classes of a congruence on the weak order on $S_n$.
\end{prop}

\begin{prop}\label{rho t avoid}

\noindent
\begin{enumerate}
\item A permutation is minimal in its $\rho_t$-fiber if and only if it avoids $3$-$1$-$2$.
\item A permutation is maximal in its $\rho_t$-fiber if and only if it avoids $1$-$3$-$2$.
\end{enumerate}
\end{prop}

Modifying the definitions above in the obvious ways to obtain the definitions of \emph{$(312\to132)$-moves}, \emph{$(132\to312)$-moves}, $\Theta_{312}$, $\pidown^{312}$, and $\piup_{312}$, we also have the following results.

\begin{prop}\label{rho t move}
Let $x\covered y$ in the weak order.
Then $\rho_t(x)=\rho_t(y)$ if and only if $x$ is obtained from $y$ by a $(312\to132)$-move.
\end{prop}

\begin{prop}\label{pidown 312 char}
Let $y\in S_n$.
If some permutation $x\in S_n$ can be obtained from $y$ by a $(312\to132)$-move, then $\pidown^{312}(y)=\pidown^{312}(x)$.
Otherwise, $\pidown^{312}(y)=y$. 
\end{prop}

\begin{prop}\label{piup 312 char}
Let $x\in S_n$.
If some permutation $y\in S_n$ can be obtained from $x$ by a $(132\to312)$-move, then $\piup_{312}(x)=\piup_{312}(y)$.
Otherwise, $\piup_{312}(x)=x$. 
\end{prop}

In particular, LR is the Hopf algebra of ($3$-$1$-$2$)-avoiding permutations.
The quotient of the weak order modulo $\Theta_{312}$ is the restriction of the weak order to ($3$-$1$-$2$)-avoiding permutations.
This quotient is also isomorphic to the Tamari lattice, is also a sublattice of the weak order, and plays a role for LR analogous to the role the weak order plays for MR.

The characterization of LR and LR$'$ in terms of pattern-avoidance now makes it easy to prove the earlier assertion that they are anti-isomorphic.
Recall from Section~\ref{auto sec} that the map $\rv\circ\rp$ is an antiautomorphism of MR.
One easily verifies that $\rv\circ\rp$ maps a $(231\to213)$-move to a $(312\to132)$-move.
Thus by Propositions~\ref{rho b move} and~\ref{rho t move} (which completely characterize the fibers of $\rho_b$ and $\rho_t$ because congruence classes are intervals), we see that $\rv\circ\rp$ maps the fibers of $\rho_b$ to the fibers of $\rho_t$.
Thus the involution $\rv\circ\rp$ restricts to an anti-isomorphism from LR$'$ to LR.

\subsection{The Hopf algebra NSym}\label{NSym sec}
Just as LR and LR$'$ were defined using maps from permutations to planar binary trees, the Hopf algebra of subsets is also defined using a map.
In this case, we map to pairs $(U,n)$ such that $n\ge 0$ and $U$ is a set of positive integers strictly less than $n$.
In particular, $U$ is empty when $n=0$ or $n=1$. 
Given a permutation $x\in S_n$, define $\sigma(x)=(U(x),n)$, where $U(x)$ is the set of positive integers $a<n$ such that $a+1$ occurs somewhere before $a$ in $x_1x_2\cdots x_n$.

The map $\sigma$ restricts to a lattice homomorphism from the weak order on $S_n$ to the lattice of subsets of $[n-1]$.
To see why, recall a convenient characterization of the meet and join operations on the weak order.
A permutation $x\in S_n$ is uniquely determined by its \emph{(left) inversion set}: the set of pairs $(a,b)$ such that $a>b$ and $a$ occurs somewhere before $b$ in $x_1x_2\cdots x_n$.
Interpreting this set of pairs as a relation, the inversion set of $x\join y$ is the transitive closure of the union of the inversion sets of $x$ and $y$.
The set $U(x)$ is the intersection of the inversion set of $x$ with the set $\set{(a+1,a):a\in[n-1]}$.
It is thus easily seen that $U(x\join y)=U(x)\cup U(y)$.
The meet in the weak order has a similar characterization involving the transitive closure of the union of sets of ``non-inversions,'' and we easily conclude that $U(x\meet y)=U(x)\cap U(y)$.

Since $\sigma$ is a lattice homomorphism, the fibers of $\sigma$ are a lattice congruence, which we represent by the symbol $\Theta_{\operatorname{sub}}$.
This congruence is the join of $\Theta_{231}$ and $\Theta_{312}$, as we now explain.
The \emph{congruence lattice} of a lattice $L$ is the set of all congruences, partially ordered as an induced subposet of the lattice of set partitions of $L$.
In fact, this subposet is a sublattice of the lattice of set partitions, and furthermore, is a distributive lattice.
(See, for example, \cite[Theorem~II.3.11]{Gratzer}.) 
The following fact is immediate from the definition of $\sigma$:

\begin{prop}\label{sigma move}
Let $x\covered y$ in the weak order, so that $x$ is obtained from $y$ by swapping a pair of entries in adjacent positions.
Then $\sigma(x)=\sigma(y)$ if and only if the two entries differ in value by two or more.
\end{prop}

Proposition~\ref{sigma move} can be restated as follows:  
If $x\covered y$ in the weak order, then $\sigma(x)=\sigma(y)$ if and only $x$ is obtained from $y$ by either a $(231\to213)$-move or a $(312\to132)$-move.
Comparing with Propositions~\ref{rho b move} and~\ref{rho t move}, and keeping in mind that $\Theta_{\operatorname{sub}}$-classes are intervals, we conclude that $\Theta_{\operatorname{sub}}$ is the join of $\Theta_{231}$ and $\Theta_{312}$ in the lattice of partitions of $S_n$, and therefore also in the congruence lattice of the weak order on $S_n$.

The map $\sigma$ shares the remarkable Hopf-algebraic properties of $\rho_b$ and $\rho_t$: 
The sums over the fibers of $\sigma$ are a graded basis for a graded Hopf subalgebra of MR.
The Hopf subalgebra, denoted here by NSym, is better known as the Hopf algebra of noncommutative symmetric functions~\cite{GKLLRT}.
The lattice of subsets play the analogous role for NSym as the weak order and Tamari lattice play for MR and LR.
Alternately, the basis elements of NSym can be indexed by permutations avoiding both $2$-$3$-$1$ and $3$-$1$-$2$, or equivalently, permutations such that each entry is at most one greater than the entry to its right.
The considerations at the end of Section~\ref{pattern sec} show that the map $\rv\circ\rp$ is an antiautomorphism of NSym.
The same argument that shows that $\Theta_{\operatorname{sub}}=\Theta_{231}\join\Theta_{312}$ also shows that NSym is the intersection of LR and LR$'$, as Hopf subalgebras of MR.

\subsection{Some quotients of MR}\label{quot sec}
Passing to graded duals turns Hopf subalgebra relationships into quotient relationships.
Thus the self-duality of MR means that, for each Hopf subalgebra of MR, there is a corresponding quotient Hopf algebra of MR (or, in other words, a homomorphic image of MR). 
In particular, the duals LR$^*$, $($LR$')^*$ and NSym$^*$ are quotients of MR$\cong$MR$^*$.

Consider the map LR$\into$MR$\stackrel{\cong}{\longleftrightarrow}$MR$^*\onto$LR$^*$, where the left arrow is the inclusion defined in Section~\ref{lr sec}, the right arrow is its dual surjection, and the middle arrow is the isomorphism $\inv$.
This is a homomorphism of graded Hopf algebras, because each of its three steps is a homomorphism.
It is also an invertible map of graded vector spaces~\cite{Quelques,HNT}, and thus it is an isomorphism.
Therefore LR is self-dual.  
The Hopf algebra NSym is not self-dual; its dual NSym$^*$ is better known as the Hopf algebra of quasi-symmetric functions.

\section{The Hopf algebra tBax of twisted Baxter permutations}\label{tBax Hopf sec}
The realizations, given in Section~\ref{sub Hopf sec}, of the Hopf subalgebras LR, LR$'$ and NSym of MR all ultimately rest on natural combinatorial maps.
Furthermore, these maps are lattice homomorphisms from the weak order to natural lattice structures.
In~\cite{con_app}, a unification and generalization of this construction is given.
We begin this section with a brief, undetailed summary of the approach and results of~\cite{con_app}.
We then describe in more detail a particular new Hopf algebra which arose from the generalization.

\subsection{$\H$-Families}\label{H fam sec}
For each $n\ge 0$, let $\Theta_n$ be a lattice congruence on the weak order on $S_n$.
Let $(\pidown)_n$ be the map taking each $x\in S_n$ to the minimal representative of its $\Theta_n$-class.
Let $Z^\Theta_n=(\pidown)_n(S_n)$ be the set of minimal elements of $\Theta_n$-classes.
(Thus, by a general theorem mentioned earlier, the weak order restricted to $Z^\Theta_n$ is isomorphic to the lattice quotient of $S_n$ modulo $\Theta_n$.)
Define a graded vector space $\K[Z_\infty^\Theta]=\bigoplus_{n\ge 0}\K[Z^\Theta_n]$.
For each fixed choice of a sequence of congruences, define a map $c^\Theta:\K[Z^\Theta_\infty]\to\K[S_\infty]$ by sending each element of $Z^\Theta_n$ to the sum (with coefficient 1) over its $\Theta_n$-class.
Define another map $r^\Theta:\K[S_\infty]\to\K[Z^\Theta_\infty]$ which, for each $x\in S_n$, fixes $x$ if $x\in Z^\Theta_n$ and otherwise sends $x\in S_n$ to zero.
The letters $c$ and $r$ suggest ``class'' and ``representative.''
The map $c$ is an injection, while $r$ is a surjection, and the composition $r\circ c$ is the identity on $\K[Z^\Theta_\infty]$, while $c\circ r$ is the identity on $c^\Theta(\K[Z^\Theta_\infty])$.

The graded vector spaces LR, LR$'\!$, and NSym can all be realized as $\K[Z^\Theta_\infty]$ for choices of $\Theta_n$ described in Section~\ref{sub Hopf sec}.
Moreover, in each case, the elements of the form $c^\Theta(z)$ form a graded basis of a graded Hopf subalgebra of MR.
The main result of the last half of~\cite{con_app} is a characterization of families $\set{\Theta_n}_{n\ge 0}$ of congruences having this same property.
A family $\set{\Theta_n}_{n\ge 0}$ is said to be an $\H$-family if it satisfies certain conditions given in~\cite{con_app}.
If $\set{\Theta_n}_{n\ge 0}$ is an $\H$-family, then the elements $c^\Theta(z)$ form a graded basis of a graded Hopf subalgebra of MR.
In this case, $\K[Z^\Theta_\infty]$ has a unique Hopf algebra structure such that $c^\Theta$ is an injective homomorphism of Hopf algebras.
The product $x\bullet_Zy$ is $r^\Theta(x\bullet_Sy)$.
In other words, $x\bullet_Zy$ is obtained by summing over all shifted shuffles of $x$ and $y$, but ignoring those shuffles that are not in some $Z^\Theta_n$.
The coproduct is $\Delta_Z=(r^\Theta\otimes r^\Theta)\circ\Delta_S\circ c^\Theta$.
In other words, the coproduct has the following formula, which corrects an unfortunate typographical error\footnote{In Equation (2) of \cite[Section~8]{con_app}, the symbols $x_1,\ldots,x_p$ and $x_{p+1},\ldots,x_n$ should be replaced by $y_1,\ldots,y_p$ and $y_{p+1},\ldots,y_n$.} of Equation (2) of \cite[Section~8]{con_app}.
\[\Delta_Z(x)=\sum_{\substack{y\in S_n\\\pidown y=x}}\sum_{i=0}^n\std^\Theta(y_1,\ldots,y_i)\otimes\std^\Theta(y_{i+1},\ldots,y_n).\]
Here $\pidown y$ is short for $(\pidown)_n(y)$, and $\std^\Theta$ is the composition $r\circ\std$.
In other words, $\std^\Theta$ standardizes the sequence, and sends the standardization to zero if it is not in some $Z^\Theta_n$.

For any $\H$-family $\set{\Theta_n}_{n\ge 0}$, the elements of $Z_n^\Theta$ are described by a pattern-avoidance condition given explicitly in \cite[Theorem~9.3]{con_app}.
Special cases of the condition have been seen in Sections~\ref{pattern sec} and~\ref{NSym sec}.
Another special case is described in Section~\ref{tBax sec}.
Each map $(\pidown)_n$ is a lattice homomorphism from the weak order on all of $S_n$ to the weak order restricted to $Z^\Theta_n$, the set of permutations satisfying the pattern-avoidance condition.

\subsection{Twisted Baxter permutations}\label{tBax sec}
We now define twisted Baxter permutations in the generalized pattern notation of~\cite{Gen Pat}.
An occurrence of the pattern $3$-$41$-$2$ in a permutation $x=x_1x_2\cdots x_n$ is a subsequence $cdab$ of $x_1x_2\cdots x_n$ with $a<b<c<d$ such that $d$ occurs immediately previous to $a$ in $x_1x_2\cdots x_n$.
This differs from an occurrence of the pattern $3$-$4$-$1$-$2$ in $x$, in the sense defined in Section~\ref{pattern sec}, which is any subsequence $cdab$ of $x_1x_2\cdots x_n$ with $a<b<c<d$, with no requirement that $d$ and $a$ be adjacent.
Thus for example, the pattern $3$-$4$-$1$-$2$ occurs in $45312\in S_5$ but the pattern $3$-$41$-$2$ does not.
Similarly, one can define $2$-$41$-$3$-patterns, $3$-$14$-$2$-patterns, and $2$-$14$-$3$-patterns.

\emph{Twisted Baxter permutations} are permutations that avoid both $3$-$41$-$2$ and \mbox{$2$-$41$-$3$}.
They were introduced in~\cite{con_app} as the basis of a new Hopf algebra arising from the construction described in Section~\ref{H fam sec}.
They were conjectured in~\cite{con_app} to be in bijection with the \emph{(reduced) Baxter permutations} of~\cite{CGHK}, and this conjecture was soon after proved by West~\cite{West pers} using generating trees.
In Section~\ref{Bax tBax bij}, we provide a different proof.
A permutation is a Baxter permutation if and only if it avoids both $3$-$14$-$2$ and $2$-$41$-$3$.
Notice that the definition of the twisted Baxter permutations is a very slight modification of the definition of Baxter permutations.
Let $\tBax_n$ be the set of twisted Baxter permutations in $S_n$ and let $\Bax_n$ be the set of Baxter permutations in $S_n$.

The patterns $3$-$41$-$2$, $2$-$41$-$3$, $3$-$14$-$2$, and $2$-$14$-$3$ have the following special property:
\begin{prop}\label{23 adj}
If a permutation $x$ contains the pattern $3$-$41$-$2$, then $x$ contains an instance of the pattern $3$-$41$-$2$ such that the ``3'' and the ``2'' differ in value by exactly 1.
The same is true of the patterns $2$-$41$-$3$, $3$-$14$-$2$, and $2$-$14$-$3$.
\end{prop}
The assertion of the proposition is that if $x=x_1\cdots x_n$ contains $3$-$41$-$2$ then $x_1\cdots x_n$ contains a subsequence $(b+1)dab$ with $a<b<b+1<d$ and $d$ adjacent to $a$ in $x_1\cdots x_n$.
\begin{proof}
Suppose $cdab$ is a subsequence of $x_1\cdots x_n$ with $a<b<c<d$, such that $d$ is adjacent to $a$.
Furthermore, among all such subsequences, choose $cdab$ to minimize $c-b$.
Suppose for the sake of contradiction that $c\neq b+1$.
If $b+1$ occurs to the left of $d$, then $(b+1)dab$ is an instance of the pattern $3$-$41$-$2$, contradicting our choice of $cdab$.
Otherwise $b+1$ occurs to the right of $a$, and $cda(b+1)$ is an instance of the pattern $3$-$41$-$2$, again contradicting our choice of $cdab$.
This contradiction in both cases proves that $c=b+1$.
We have proved the statement for the pattern $3$-$41$-$2$.
The proof for the other patterns is essentially identical.
\end{proof}

Proposition~\ref{23 adj} has the following corollary, which explains why the definition above appears to define the inverses of the Baxter permutations defined in~\cite{CGHK}.  

\begin{cor}\label{inv Bax}
A permutation $x$ is a Baxter permutation if and only if the inverse permutation $x^{-1}$ is a Baxter permutation.
\end{cor}
\begin{proof}
If $x$ contains the pattern $3$-$14$-$2$, then by Proposition~\ref{23 adj}, $x$ contains a subsequence $(b+1)adb$ with $a<b<b+1<d$ and $a$ adjacent to $d$.
Write $x^{-1}=y_1\cdots y_n$.
Then the subsequence $y_ay_by_{b+1}y_d$ of $y_1\cdots y_n$ is an instance of the pattern $2$-$41$-$3$.
If $x$ contains $2$-$41$-$3$, then by Proposition~\ref{23 adj}, $x$ contains a subsequence $bda(b+1)$ with $a<b<b+1<d$ and $d$ adjacent to $a$.
In this case, the subsequence $y_ay_by_{b+1}y_d$  is an instance of the pattern $3$-$14$-$2$.
In either case, $x^{-1}$ is not a Baxter permutation.
\end{proof}


A further property of twisted Baxter permutations and Baxter permutations is helpful.
Given a permutation $x=x_1\cdots x_n\in S_n$ and a subset $V\subseteq [n]$, the \emph{restriction of $x$ to the values $V$} is the subsequence of $x_1\cdots x_n$ consisting of entries in $V$.
Note that this is a restriction of $x$, thought of as a sequence, not a restriction of the map $x:[n]\to[n]$.

\begin{prop}\label{tBax restrict}
If $x\in \tBax_n$ and $0\le p\le n$, then the restriction of $x$ to the values $1,\ldots,p$ is in $\tBax_p$, and the standardization of the restriction of $x$ to the values $p+1,\ldots,n$ is in $\tBax_{n-p}$.
\end{prop}
\begin{proof}
To prove the first assertion, suppose to the contrary that the restriction of $x$ to $1,\ldots,p$ contains a subsequence $abcd$ which is an instance either of the pattern $3$-$41$-$2$ or the pattern $2$-$41$-$3$.
Since $x$ avoids both $3$-$41$-$2$ and $2$-$41$-$3$, there are entries of $x$ between $b$ and $c$.
These entries are all $>p$.
If $e$ is the rightmost such entry, then $aecd$ is either a $3$-$41$-$2$-pattern or a $2$-$41$-$3$-pattern in $x$, and this contradiction proves the assertion.
The second assertion is proved similarly, except that we choose $e$ to be the leftmost entry between $b$ and $c$ and conclude that $abed$ is a forbidden pattern in $x$.
\end{proof}

An essentially identical argument proves the following proposition.

\begin{prop}\label{Bax restrict}
If $x\in \Bax_n$ and $0\le p\le n$, then the restriction of $x$ to the values $1,\ldots,p$ is in $\Bax_p$, and the standardization of the restriction of $x$ to the values $p+1,\ldots,n$ is in $\Bax_{n-p}$.
\end{prop}

\subsection{The Hopf algebra tBax}\label{tBax Hopf subsec}
The definition of twisted Baxter permutations is a special case of the general pattern-avoidance condition of \cite[Theorem~9.3]{con_app}.  
Thus, the twisted Baxter permutations are the minimal congruence class representatives of an $\H$-family of congruences.
We refer to the congruences in this family collectively as $\Theta_{\tB}$.
The associate downward projection map, taking $x\in S_n$ to the minimal element of its $\Theta_{\tB}$-class, is denoted by $\pidown^{\tB}$, and the analogous upward projection map is $\piup_{\tB}$.
Let $y\in S_n$.
We say that $x$ is obtained from $y$ by a \emph{$(3412\to3142)$-move} if there exists an instance $cdab$ of the pattern $3$-$41$-$2$ in $y$ and if $x$ is obtained from $y$ by exchanging $d$ and $a$.
In this case, in particular $x\covered y$.  
We also say, in this case, that $y$ is obtained from $x$ by a \emph{$(3142\to3412)$-move}.
The notions of \emph{$(2413\to2143)$-moves} and \emph{$(2143\to2413)$-moves} are defined analogously.

As a special case of results of~\cite[Section~9]{con_app}, the congruences $\Theta_{\tB}$ have the following properties, the first of which is a repetition for emphasis.

\begin{prop}\label{tBax props}

\noindent
\begin{enumerate}
\item \label{tBax bottoms}
A permutation is the minimal element in its $\Theta_{\tB}$-class if and only if it is a twisted Baxter permutation (\emph{i.e.}, if and only if it avoids $3$-$41$-$2$ and $2$-$41$-$3$). 
\item A permutation is the maximal element in its $\Theta_{\tB}$-class if and only if it avoids $3$-$14$-$2$ and $2$-$14$-$3$.
\item \label{tBax moves}
Suppose $x\covered y$ in the weak order.  Then $x\equiv y$ modulo $\Theta_{\tB}$ if and only if $x$ is obtained from $y$ by either a $(3412\to3142)$-move or a $(2413\to2143)$-move.
\item Let $y\in S_n$.
If some permutation $x\in S_n$ can be obtained from $y$ by a $(3412\to3142)$-move or a $(2413\to2143)$-move, then $\pidown^{\tB}(y)=\pidown^{\tB}(x)$.
Otherwise, $\pidown^{\tB}(y)=y$. 
\item \label{tBax piup}
Let $x\in S_n$.
If some permutation $y\in S_n$ can be obtained from $x$ by a $(3142\to3412)$-move or a $(2143\to2413)$-move, then $\piup_{\tB}(x)=\piup_{\tB}(y)$.
Otherwise, $\piup_{\tB}(x)=x$.
\end{enumerate}
\end{prop}

The twisted Baxter permutations constitute a graded basis for a Hopf subalgebra tBax of MR.
As a special case of the description in Section~\ref{H fam sec}, the product and coproduct in tBax have the following descriptions.
The product $x\bullet_{\tB}y$ of two twisted Baxter permutations $x$ and $y$ is the sum of all twisted Baxter permutations that are obtained as shifted shuffles of $x$ and $y$.
Thus, for example, 
\begin{align*}
21\bullet_{\tB} 21&=2143+\som{2413} +2431+4213+4231+4321\\
	&=2143+2431+4213+4231+4321\\
21\bullet_{\tB} 132&=21354 + 23154+\som{23514}+23541\\& \qquad+ \,32154 +\som{32514}+32541+\som{35214}+\som{35241}+35421\\
&=21354 + 23154+23541+ 32154 +32541+35421
\end{align*}

For a twisted Baxter permutation $z$, the coproduct $\Delta_\tB(z)$ is obtained in two steps.
First sum $\Delta_S(w)$ over all permutations $w$ in the $\Theta_\tB$-class of $z$.
Then delete all terms which contain permutations that are not twisted Baxter.
For example, let $z=532641\in\tBax_6$.  
By Proposition~\ref{tBax props}, we calculate the $\Theta_\tB$-class of $z$ to be $\set{532641,536241, 563241}$.
\begin{align*}
\Delta_S(563241)&=\emptyset\otimes563241+1\otimes53241+12\otimes3241\\
	&\qquad+\,231\otimes231+3421\otimes21+45213\otimes1+563241\otimes\emptyset\\
\Delta_S(536241)&=\emptyset\otimes536241+1\otimes35241+21\otimes4231\\
	&\qquad+\,213\otimes231+3241\otimes21+42513\otimes1+536241\otimes\emptyset\\
\Delta_S(532641)&=\emptyset\otimes532641+1\otimes32541+21\otimes2431\\
	&\qquad+\,321\otimes321+3214\otimes21+42153\otimes1+532641\otimes\emptyset.
\end{align*}
Thus, 
\begin{align*}
\Delta_\tB(532641)&=\som{\emptyset\otimes563241} + 1\otimes53241 + 12\otimes3241 + 231\otimes231 + 3421\otimes21\\
	&\qquad +\, \som{45213\otimes1} +  \som{563241\otimes\emptyset} + \som{\emptyset\otimes536241} + \som{1\otimes35241}\\
	&\qquad + \,21\otimes4231 + 213\otimes231 + 3241\otimes21 + \som{42513\otimes1} \\
	&\qquad+ \,\som{536241\otimes\emptyset} +\emptyset\otimes532641 + 1\otimes32541 + 21\otimes2431 \\
	&\qquad+ \,321\otimes321 + 3214\otimes21 + 42153\otimes1 + 532641\otimes\emptyset \\
&=\emptyset\otimes532641 + 1\otimes(53241+32541)  + 12\otimes3241 \\
	&\qquad+\,21\otimes(4231+2431)  + (213+231)\otimes231 + 321\otimes321 \\
	&\qquad+ (3421+ 3241+3214)\otimes21 + 42153\otimes1 + 532641\otimes\emptyset	
\end{align*}

Notice that all terms in $\Delta_\tB(532641)$ occur with coefficient 1.
This turns out to be true in general.
\begin{prop}\label{perm coeff 1}
Let $z\in\tBax_n$.
Then each term occurring with nonzero coefficient in $\Delta_\tB(z)$ occurs with coefficient 1.
\end{prop}
\begin{proof}
Suppose $v$ and $w$ are distinct permutations in the $\Theta_\tB$-class of $z$.
It is enough to show that, for any $p$ with $0\le p\le n$, either $\std(v_1,\ldots,v_p)\neq\std(w_1,\ldots,w_p)$ or $\std(v_{p+1},\ldots,v_n)\neq\std(w_{p+1},\ldots,w_n)$ or both.

Fix such a $p$ and let $y=v\meet w$, the meet of $v$ and $w$ in the weak order.
Since $\Theta_\tB$ is a lattice congruence, $y$ is also in the $\Theta_\tB$-class of $z$.
Choose $i$ to be the smallest positive index such that either $v_i$ differs from $y_i$ or $w_i$ differs from $y_i$ or both.
Since $v$ and $w$ are distinct, such an $i$ exists and is less than $n$.
We claim that $v_i\neq w_i$.

Suppose to the contrary that $v_i=w_i$ and let $k$ be such that $y_k=v_i$.
Then $k>i$, because $v$ and $y$ agree in positions left of position $i$.
For every $j$ with $i\le j<k$, the element $y_j$ precedes the entry $v_i$ in $y$ but follows $v_i$ in $v$.
Since $y\le v$ in the weak order, and since the weak order is the containment order on inversion sets, we conclude that $y_j<v_i=y_k$.
Let $y'$ be the permutation $y_1\cdots y_{i-1}y_ky_i\cdots y_{k-1}y_{k+1}\cdots y_n$, obtained by removing $y_k$ from $y$ and re-inserting it between $y_{i-1}$ and $y_i$.
Passing from $y$ to $y'$ creates some inversions, but destroys no inversions, so $y'>y$.
The inversions created in $y'$ are inversions of $v$, so $y'\le v$.
By the same argument, $y'\le w$.
This contradicts our choice of $y$ to be the meet of $v$ and $w$, thus proving the claim.

If $p<i$ then $\std(v_{p+1},\ldots,v_n)\neq\std(w_{p+1},\ldots,w_n)$ because $v$ and $w$ are distinct, so assume $p\ge i$.
By the claim, we take, without loss of generality, $v_i>w_i$.
We complete the proof by showing that there exists $l$ with $1\le l<i$ such that $w_i<v_l<v_i$.
Since $w_l=v_l$, this implies that $\std(v_1,\ldots,v_p)\neq\std(w_1,\ldots,w_p)$.

Since $(v_i,w_i)$ is not an inversion in $w$, it is also not an inversion in $y$.
Consider a saturated chain in the weak order from $y$ to $v$.
Since $(v_i,w_i)$ is not an inversion in $y$ but is an inversion of $v$, as we follow the chain from $y$ up to $v$, we eventually come to an edge $x\covered x'$ such that $x$ and $x'$ differ only in the order of the adjacent entries $v_i$ and $w_i$.
Since $\Theta_\tB$-classes are intervals, the entire chain is contained in the same $\Theta_\tB$-class, so Proposition~\ref{tBax props}.\ref{tBax moves} says in particular that there exists an entry $a$ to the left of $w_i$ and $v_i$ in $x$ and $x'$ with $w_i<a<v_i$.
Choose $a$ to be the largest entry to the left of $w_i$ and $v_i$ with $w_i<a<v_i$.
If $a$ is not $v_l$ for some $1\le l<i$, then as we proceed further up the chain from $x'$ to $v$, we eventually come to an edge $u\covered u'$ corresponding to swapping adjacent entries $a$ and $v_i$.
But then Proposition~\ref{tBax props}.\ref{tBax moves} says that there exists an entry $a'$ to the left of $a$ and $v_i$ in $u$ and $u'$ with $a<a'<v_i$.
Since $u$ is weakly higher up the chain than $x'$ and since $(v_i,a')$ is not an inversion in $u$, we conclude that $(v_i,a')$ is not an inversion in $x$ and $x'$.
Thus $a'$ is to the left of $v_i$ in $x$ and $x'$, contradicting our choice of $a$ to be maximal.
This contradiction shows that $a=v_l$ for some $1\le l<i$, thus completing the proof.
\end{proof}

The Hopf algebra tBax is closely related to LR and LR$'$.
Recall from Section~\ref{NSym sec} that the congruence $\Theta_{\operatorname{sub}}$, which defines the Hopf algebra NSym, is the join of the congruences $\Theta_{231}$ and $\Theta_{312}$ in the lattice of partitions of $S_n$, and thus in the lattice of congruences of the weak order on $S_n$.
The congruence $\Theta_\tB$, which defines the Hopf subalgebra tBax, is just the opposite.
The following proposition is \cite[Proposition~10.2]{con_app}.
\begin{prop}\label{10.2}
The congruence $\Theta_\tB$ is the meet of $\Theta_{231}$ and $\Theta_{312}$ in the lattice of congruences of the weak order on $S_n$ (or equivalently in the lattice of partitions of $S_n$).
\end{prop}
That is, two permutations are congruent modulo $\Theta_\tB$ if and only if they are congruent modulo $\Theta_{231}$ and modulo $\Theta_{312}$.

\begin{remark}\label{smallest sub Hopf}
Proposition~\ref{10.2} is a weaker statement than the statement that tBax is the smallest Hopf subalgebra of MR containing LR and LR$'$.
However, we have no evidence against the stronger statement.
\end{remark}

\begin{remark}\label{tBax auto}
Several automorphisms and antiautomorphisms of tBax are inherited from MR.
The algebra automorphism/coalgebra antiautomorphism $\rp$ of MR takes $(3412\to3142)$-moves to $(2143\to2413)$-moves and takes $(2413\to2143)$-moves to $(3142\to3412)$-moves, and thus, by Proposition~\ref{tBax props}.\ref{tBax moves}, restricts to an algebra automorphism/coalgebra antiautomorphism of tBax.
The algebra antiautomorphism/coalgebra automorphism $\rv$ of MR has the same effect on $(3412\to3142)$-moves and $(2413\to2143)$-moves, and thus restricts to an algebra antiautomorphism/coalgebra automorphism of tBax.
Their composition $\rv\circ\rp$ is a Hopf algebra antiautomorphism of tBax.
\end{remark}

\begin{remark}\label{tBax dual}
Recall from Section~\ref{quot sec} that the Hopf algebra LR is self-dual, and that its self-duality can be proved by verifying that the Hopf algebra homomorphisms LR$\into$MR$\stackrel{\cong}{\longleftrightarrow}$MR$^*\onto$LR$^*$ is an invertible map of graded vector spaces.
We remark that the analogous argument fails when applied to tBax.
That is, the map tBax$\into$MR$\stackrel{\cong}{\longleftrightarrow}$MR$^*\onto$tBax$^*$ is not invertible.
For example, one can verify that the matrix of the restriction of this map to $\tBax_4$ has two identical rows.
Of course, the failure of this argument does not rule out the existence of some other isomorphism between tBax and tBax$^*$.
\end{remark}

\section[The Hopf algebra dRec of diagonal rectangulations]{The Hopf algebra of diagonal rectangulations}\label{rec sec}
As described above, the Hopf algebra of twisted Baxter permutations has, like LR, LR$'$, and NSym, a realization in terms of pattern avoiding permutations.
The product and coproduct are calculated by computing products and coproducts in MR, and then deleting terms which don't satisfy the avoidance conditions.
However, LR,  LR$'$, and NSym admit intrinsic descriptions of their products and coproducts in terms of the fundamental combinatorial objects (planar binary trees and subsets) indexing their bases, without first embedding into a larger Hopf algebra \cite{LRorder,GKLLRT}.
In this section, we supply a similar combinatorial realization of the Hopf algebra tBax.
Specifically, we define diagonal rectangulations and define a product and coproduct on the graded vector space spanned by diagonal rectangulations.
The product and coproduct on diagonal rectangulations define a Hopf algebra, which we call dRec.
The proof that dRec is a Hopf algebra appears in Section~\ref{isom sec}, where it is shown that dRec is isomorphic to the Hopf algebra tBax.

\subsection{Diagonal rectangulations}\label{diag rec sec}
The general notion of a rectangulation of a rectangle arises in integrated circuit design.  
(See references in \cite[Section~1]{ABP}.)
We are interested in a particular class of rectangulations of a square which we call \emph{diagonal rectangulations}.
Ackerman, Barequet, and Pinter \cite[Lemma~5.2]{ABP} showed that diagonal rectangulations are counted by the Baxter number, by showing that they satisfy a recurrence from~\cite{CGHK}.
(A proof using nonintersecting lattice paths is given in \cite{FFNO}.)
Thus diagonal rectangulations emerge as a candidate to index the basis of a Hopf algebra isomorphic to tBax.

Let $S$ be a square with sides parallel to the coordinate axes.  
Consider a subdivision of $S$ into finitely many rectangles.
Necessarily, all the sides of the rectangles are also parallel to the coordinate axes. 
We define the \emph{walls} of the decomposition to be the maximal line segments contained in the union of the boundaries of the rectangles, excepting the four edges of $S$, which are not considered to be walls of the decomposition.
The interior of the line segment connecting the top-left corner of~$S$ to the bottom-right corner is called the \emph{diagonal} of $S$.
Let $X$ be a set of $n-1$ distinct points on the diagonal of $S$.  
A diagonal rectangulation of $(S,X)$ is a subdivision of $S$ into rectangles such that no two walls intersect in their interiors, such that  every wall of the decomposition contains a point of $X$, and such that every point of $X$ lies on a wall.  

\begin{prop}\label{n n-1}
A diagonal rectangulation of $(S,X)$ with $|X|=n-1$ has exactly $n-1$ walls and has exactly $n$ rectangles.
\end{prop}
\begin{proof}
Let $R$ be a diagonal rectangulation.
Because the points of $X$ are distinct points on the diagonal, every wall defining $R$ contains exactly one point in $X$.
Suppose two walls contain the same point $x$ in $X$.
Since the walls have disjoint interiors, $x$ is an endpoint of both.
By the definition of the walls of $R$, the two walls are not contained in the same line, and in particular, there are only two walls containing $x$.
But now, considering $R$ locally about $x$, we see that $R$ is not a decomposition into rectangles, and this contradiction shows that every point in $X$ is on exactly one wall of $R$.
Thus $R$ has $n-1$ walls.

We conclude the proof by arguing that any decomposition of $S$ into rectangles with $n-1$ interior-disjoint walls has exactly $n$ rectangles.
Indeed, let $L$ be any wall in such a decomposition.
There are two walls (or edges of $S$) whose interiors contain the endpoint of $L$, and some walls $L_1,\ldots,L_k$ whose endpoints are contained in the interior of $L$.
The situation is illustrated in Figure~\ref{wall}.a, where $L$ is the red (or gray) wall.
\begin{figure}[t]
\begin{tabular}{cccc}
\scalebox{1}{\includegraphics{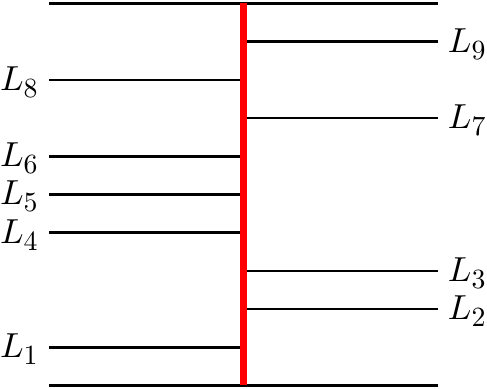}}&&&\scalebox{1}{\includegraphics{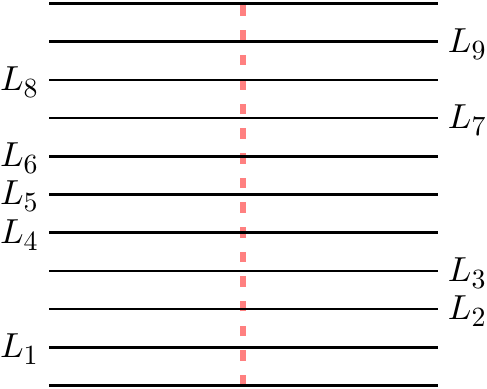}}\\
(a)&&&(b)
\end{tabular}
\caption{An illustration for the proof of Proposition~\ref{n n-1}.}
\label{wall}
\end{figure}
The number of rectangles incident to the interior of $L$ is $k+2$.
Removing $L$ and extending the walls $L_1,\ldots,L_k$, as illustrated in Figure~\ref{wall}.b, we obtain a decomposition of $S$ into rectangles with $n-2$ walls.
By induction, there are $n-1$ rectangles in the new decomposition.
The number of rectangles in the new decomposition intersecting the interior of $L$ is $k+1$, and thus the original decomposition has $n$ rectangles.  
\end{proof}

The precise size of $S$ and the exact choice of points in $X$ is irrelevant to the combinatorics of diagonal rectangulations.
Indeed, given a rectangulation $R$ of $(S,X)$, any dilation applied to $S$ can also be applied to $X$ and $R$ to give a well-defined rectangulation of the dilated square.
Furthermore, given any choice $X'$ of $n-1$ points on the diagonal of $S$, there is a natural choice of a rectangulation $R'$ of $(S,X')$ corresponding to $R$.
(This is most easily understood by considering shifting one point of $X$ at a time, without moving it past any other point of $X$.)
These considerations define an equivalence relation on the set of all diagonal rectangulations of pairs $(S,X)$, where $S$ and $X$ vary.
The fundamental combinatorial object is an equivalence class of such rectangulations.
We call such an equivalence class a \emph{diagonal rectangulation of size $n$}. 
For convenience, we blur the distinction between diagonal rectangulations of size $n$ (\emph{i.e.}\ equivalence classes) and equivalence-class representatives.
For definiteness, we often consider the \emph{integral representative} of the equivalence class.
This is the representative such that $S=[0,n]\times[0,n]$ and $X$ is the set of integer points on the diagonal.
In particular, we often specify a diagonal rectangulation by a representative picture.
Figure~\ref{rec ex} shows (the integral representative of) a diagonal rectangulation of size 20.
The diagonal is shown in gray.
\begin{figure}[t]
\centerline{\scalebox{1}{\includegraphics{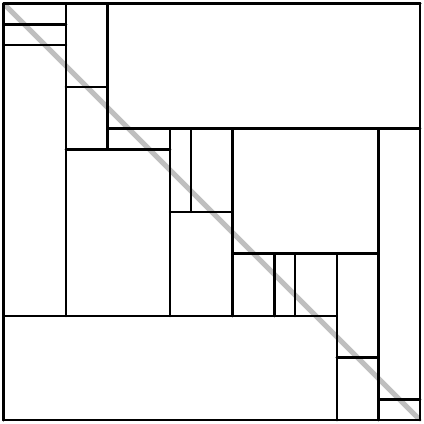}}}
\caption{A rectangulation of size 20}
\label{rec ex}
\end{figure}

An alternate characterization of diagonal rectangulations is useful. 
\begin{prop}\label{alt char}
A decomposition of a square into rectangles is a diagonal rectangulation of size $n$ if and only if it has $n$ rectangles and each rectangle's interior intersects the diagonal.
\end{prop}
\begin{proof}
Suppose $R$ is a diagonal rectangulation of size $n$.
Then Proposition~\ref{n n-1} says that $R$ has $n$ rectangles.
If one of these rectangles $U$ is contained completely below the diagonal, then consider the two walls of $R$ containing the top-right corner of~$U$.
One of these walls ends at the corner, and in particular that wall contains no point of $X$.
This contradiction, and the analogous contradiction for rectangles contained above the diagonal, complete one direction of the proof.

Conversely, suppose that $R$ is a decomposition of $S$ into $n$ rectangles, with each rectangle's interior intersecting the diagonal.
Moving along the diagonal of $S$, we visit each rectangle exactly once, and there are exactly $n-1$ points where we cross from one rectangle to another.
Let $X$ be the set of these $n-1$ points.
It is immediate that every point in $X$ is contained in a wall of $R$. 
To complete the proof, we need to show that every wall of $R$ contains a point in $X$.
Equivalently, we need to show that every wall of $R$ crosses the diagonal.
But if some wall of $R$ does not cross the diagonal, then its closest point to the diagonal is the corner of some rectangle whose interior does not intersect the diagonal.
\end{proof}

\begin{remark}\label{mosaic}
Diagonal rectangulations are in bijection with the \emph{Mosaic floorplans} defined in~\cite{HHCGDCG} in connection with a proposed floorplanning algorithm for VLSI circuit design.
Consider the class of (not necessarily diagonal) rectangulations with exactly $n-1$ walls, such that the interiors of walls are disjoint.
Mosaic floorplans of size $n$ are equivalence classes of such rectangulations under a certain equivalence relation which is coarser than topological equivalence.
In~\cite{YCCG}, mosaic floorplans are shown to be counted by the Baxter number.
As a special case of \cite[Theorem~4]{ABP}, each mosaic floorplan has exactly one representative that is a diagonal rectangulation.
\end{remark}

\subsection{Product and coproduct}\label{Rec operations sec}
Let $\dRec_n$ stand for the set of rectangulations of size $n$.
The set $\dRec_0$ has a single element, the rectangulation of a $0\times0$ square having no rectangles.
This empty rectangulation is represented by the symbol $\emptyset$.
The set $\dRec_1$ also has a single element, the ``decomposition'' of a square into a single square.
The Hopf algebra dRec is the graded vector space $\K[\dRec_\infty]=\bigoplus_{n\ge 0}\K[\dRec_n]$, with the product and coproduct that we now describe.

Given a set of interior-disjoint, axis-parallel line segments in $S$, a diagonal rectangulation~$R$ is a \emph{completion} of the set of line segments if every line segment in the set is contained in a wall of $R$.
A \emph{principal subsquare} of the square $S$ is a square~$S'$ contained in $S$ such that the diagonal of $S'$ is contained in the diagonal of $S$.
Given a diagonal rectangulation $R$ of $S$ and a principal subsquare $S'$ of $S$, the \emph{restriction} $R'$ of $R$ to $S'$ is the following decomposition of $S'$ into rectangles:
Each rectangle of $R'$ is an intersection $U\cap S'$ such that $U$ is a rectangle of $R$ and $U\cap S'$ has nonempty interior.
Each rectangle of $R'$ crosses the diagonal of $S'$, and thus $R'$ is a diagonal rectangulation by Proposition~\ref{alt char}.

Let $p$ be an integer with $0\le p\le n$.
Consider the integral representative of a diagonal rectangulation $R$, as defined in Section~\ref{diag rec sec}.
(Thus $S=[0,n]\times[0,n]$ and $X$ is the set of integer points on the diagonal.)
Define the \emph{$p\th$ top-left restriction} of $R$ to be the diagonal rectangulation $\TL_p(R)$ of size $p$ obtained as the restriction of $R$ to the principal subsquare $[0,p]\times[n-p,n]$.
For an integer~$q$ with $0\le q\le n$, define the \emph{$q\th$ bottom-right restriction} of $R$ to be the diagonal rectangulation $\BR_q(R)$ of size~$q$ obtained as the restriction of $R$ to the principal subsquare $[n-q,n]\times[0,q]$.

The product in dRec, written $\bullet_\dR$ or $\bullet$\,, is described as follows:  
If $R_1\in \dRec_p$ and $R_2\in \dRec_q$, where $p+q=n$, then $R_1\bullet_\dR R_2$ is the sum over all rectangulations $R\in\dRec_n$ such that $\TL_p(R)=R_1$ and $\BR_q(R)=R_2$.
For the purpose of computations, the product $R_1\bullet_\dR R_2$ can also be described as follows.
Represent $R_1$ and $R_2$ by their integral representations on the respective squares $[0,p]\times[0,p]$ and $[0,q]\times[0,q]$, and represent rectangulations appearing in $R_1\bullet_\dR R_2$ by their integral representations on $[0,n]\times[0,n]$.
Take the $p-1$ walls of $R_1$ and translate them upwards~$q$ units, and take the $q-1$ walls of $R_2$ and translate them $p$ units to the right.
The product is the sum of those diagonal rectangulations of size $n$ which are completions of the union of the two sets of line segments, such that the $n-1$ walls of the completion are obtained by extending (if necessary) the $n-2$ segments already present and adding a new wall containing the diagonal point $(p,n-p)$.
We give two examples in Figures~\ref{prod fig 1} and~\ref{prod fig 2}, marking, for the sake of clarity, the point $(p,n-p)$ in each diagonal rectangulation in the product. 

\begin{figure}
$\begin{array}{l}
\raisebox{-5 pt}{\includegraphics{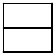}}\,\bullet\,\raisebox{-5 pt}{\includegraphics{prod_ex_1_LHS}}
=\mbox{sum of completions of}\quad\raisebox{-12 pt}{\includegraphics{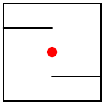}}\\[20 pt]
\qquad\qquad=
\raisebox{-12 pt}{\includegraphics{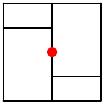}}+ 
\raisebox{-12 pt}{\includegraphics{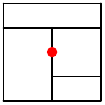}}+ 
\raisebox{-12 pt}{\includegraphics{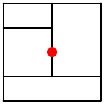}}+ 
\raisebox{-12 pt}{\includegraphics{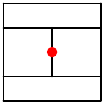}}+ 
\raisebox{-12 pt}{\includegraphics{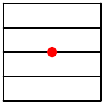}}
\end{array}
$
\caption{A product calculation in the Hopf algebra dRec}
\label{prod fig 1}
\end{figure}
\begin{figure}
$\begin{array}{l}
\raisebox{-5 pt}{\includegraphics{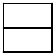}}\,\bullet\,\raisebox{-7.5 pt}{\includegraphics{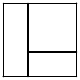}}
=\mbox{sum of completions of}\quad\raisebox{-15 pt}{\includegraphics{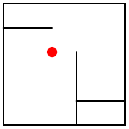}}\\[23 pt]
\quad=
\raisebox{-15 pt}{\includegraphics{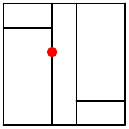}}+ 
\raisebox{-15 pt}{\includegraphics{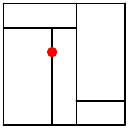}}+ 
\raisebox{-15 pt}{\includegraphics{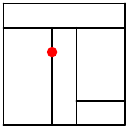}}+ 
\raisebox{-15 pt}{\includegraphics{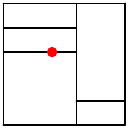}}+ 
\raisebox{-15 pt}{\includegraphics{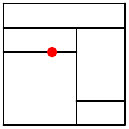}}+
\raisebox{-15 pt}{\includegraphics{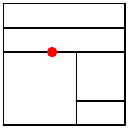}}\,.
\end{array}
$
\caption{A product calculation in the Hopf algebra dRec}
\label{prod fig 2}
\end{figure}

The coproduct on rectangulations is written $\Delta_\dR$ or simply $\Delta$.
Let $R$ be a diagonal rectangulation of size $n$ and consider any path $\gamma$ from the top-left corner of $S$ to the bottom-right corner of $S$, traveling only along edges of rectangles, and traveling only downwards and to the right.
The path $\gamma$ divides the rectangles of $R$ into two groups:  There are $p$ rectangles below/left of the path and~$q$ rectangles above/right of the path, with $p+q=n$.
If $p>0$, then by an easy induction on $p$, we conclude that there are exactly $p-1$ walls of $R$ that intersect the interior of the region in $S$ below/left of $\gamma$.
Similarly, if $q>0$, then there are exactly $q-1$ walls of $R$ that intersect the interior of the region in $S$ above/right of $\gamma$.

We associate to the path $\gamma$ an element $A_\gamma\otimes B_\gamma$ of dRec$\otimes$dRec.
If $p=0$ then the element $A_\gamma$ is $\emptyset$.
Otherwise, $A_\gamma$ is obtained as follows:  
First, delete from $R$ all parts of walls that lie on or above/right of $\gamma$.
Exactly $p-1$ line segments remain.
The element $A_\gamma$ is the sum over all completions of the collection of $p-1$ line segments to a diagonal rectangulation of size $p$.
The walls of any such completion are obtained by extending, if necessary, the line segments already present, without adding any additional walls.
Similarly, $B_\gamma=\emptyset$ if $q=0$, and otherwise $B_\gamma$ is obtained by deleting from $R$ all part of walls that lie on or below/left of $\gamma$, and then summing over all completions of the resulting collection of $q-1$ line segments to a diagonal rectangulation of size~$q$.
The coproduct of $R$ is the sum, over all paths $\gamma$, of $A_\gamma\otimes B_\gamma$.

\begin{figure}
\begin{tabular}{c|c|c}
\large$\gamma$&\large$A_\gamma$&\large$B_\gamma$\\
\hline&&\\[-5 pt]
\raisebox{-19 pt}{\includegraphics{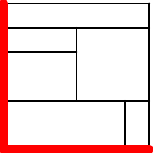}}\,\,
&$\emptyset$
&c.\,\,\raisebox{-19 pt}{\includegraphics{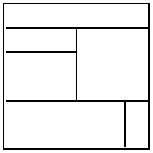}}$\,\,=\,\,$\raisebox{-19 pt}{\includegraphics{coprod_ex_B1}}\\[23 pt]
\hline&&\\[-5 pt]
\raisebox{-19 pt}{\includegraphics{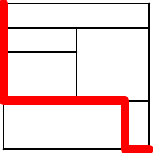}}\,\,
&c.\,\,\raisebox{-19 pt}{\includegraphics{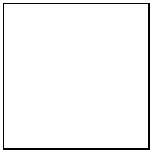}}$\,\,=\,\,$\raisebox{-1 pt}{\includegraphics{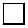}}
&\,c.\,\,\raisebox{-19 pt}{\includegraphics{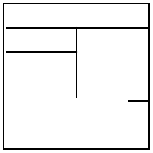}}$\,\,=\,\,$\raisebox{-13 pt}{\includegraphics{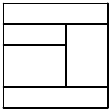}}$\,\,+\,\,$\raisebox{-13 pt}{\includegraphics{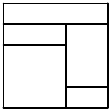}}\\[23 pt]
\hline&&\\[-5 pt]
\raisebox{-19 pt}{\includegraphics{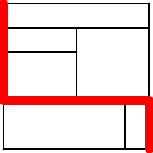}}\,\,
&c.\,\,\raisebox{-19 pt}{\includegraphics{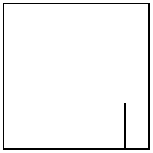}}$\,\,=\,\,$\raisebox{-3.5 pt}{\includegraphics{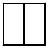}}
&c.\,\,\raisebox{-19 pt}{\includegraphics{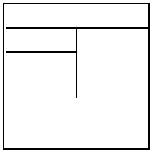}}$\,\,=\,\,$\raisebox{-9 pt}{\includegraphics{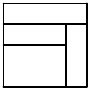}}\\[23 pt]
\hline&&\\[-5 pt]
\raisebox{-19 pt}{\includegraphics{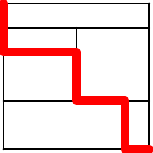}}\,\,
&c.\,\,\raisebox{-19 pt}{\includegraphics{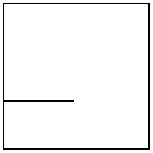}}$\,\,=\,\,$\raisebox{-3.5 pt}{\includegraphics{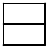}}
&c.\,\,\raisebox{-19 pt}{\includegraphics{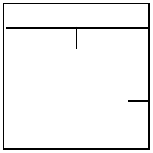}}$\,\,=\,\,$\raisebox{-9 pt}{\includegraphics{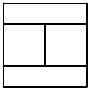}}$\,\,+\,\,$\raisebox{-9 pt}{\includegraphics{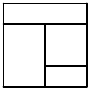}}\\[23 pt]
\hline&&\\[-5 pt]
\raisebox{-19 pt}{\includegraphics{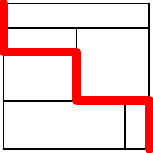}}\,\,
&c.\,\,\raisebox{-19 pt}{\includegraphics{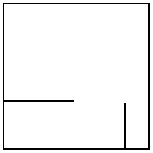}}$\,\,=\,\,$\raisebox{-6.5 pt}{\includegraphics{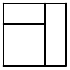}}$\,\,+\,\,$\raisebox{-6.5 pt}{\includegraphics{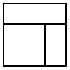}}
&c.\,\,\raisebox{-19 pt}{\includegraphics{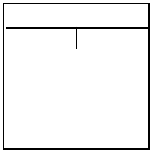}}$\,\,=\,\,$\raisebox{-6.5 pt}{\includegraphics{coprod_ex_RHS_3_2}}\\[23 pt]
\hline&&\\[-5 pt]
\raisebox{-19 pt}{\includegraphics{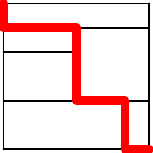}}\,\,
&c.\,\,\raisebox{-19 pt}{\includegraphics{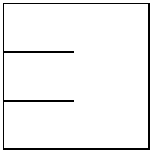}}$\,\,=\,\,$\raisebox{-6.5 pt}{\includegraphics{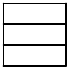}}
&c.\,\,\raisebox{-19 pt}{\includegraphics{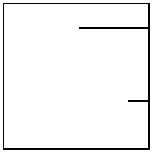}}$\,\,=\,\,$\raisebox{-6.5 pt}{\includegraphics{coprod_ex_RHS_3_3}}\\[23 pt]
\hline&&\\[-5 pt]
\raisebox{-35 pt}[0 pt][0 pt]{\includegraphics{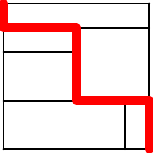}}\,\,
&c.\,\,\raisebox{-19 pt}{\includegraphics{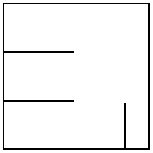}}\qquad\qquad\qquad\qquad&\\[25 pt]
&\quad$\,\,=\,\,$\raisebox{-9 pt}{\includegraphics{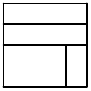}}$\,\,+\,\,$\raisebox{-9 pt}{\includegraphics{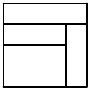}}$\,\,+\,\,$\raisebox{-9 pt}{\includegraphics{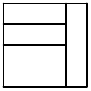}}
&\raisebox{28 pt}[0 pt][0 pt]{c.\,\,\raisebox{-19 pt}{\includegraphics{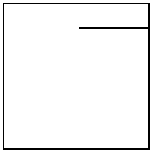}}$\,\,=\,\,$\raisebox{-3.5 pt}{\includegraphics{coprod_ex_RHS_2_2}}}\\[12 pt]
\hline&&\\[-5 pt]
\raisebox{-19 pt}{\includegraphics{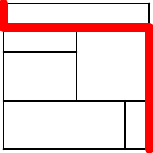}}\,\,
&c.\,\,\raisebox{-19 pt}{\includegraphics{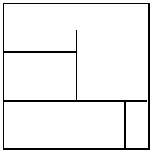}}$\,\,=\,\,$\raisebox{-13 pt}{\includegraphics{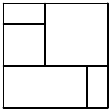}}
&c.\,\,\raisebox{-19 pt}{\includegraphics{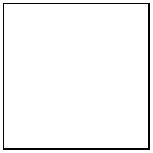}}$\,\,=\,\,$\raisebox{-1 pt}{\includegraphics{coprod_ex_RHS_1}}\\[23 pt]
\hline&&\\[-5 pt]
\raisebox{-19 pt}{\includegraphics{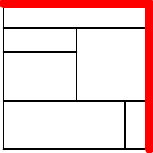}}\,\,
&c.\,\,\raisebox{-19 pt}{\includegraphics{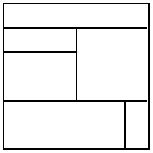}}$\,\,=\,\,$\raisebox{-19 pt}{\includegraphics{coprod_ex_B1}}
&$\emptyset$
\end{tabular}
\caption{A coproduct calculation in the Hopf algebra dRec}
\label{coprod fig}
\end{figure}
Consider, for example, the coproduct of 
\[\raisebox{-19 pt}{\includegraphics{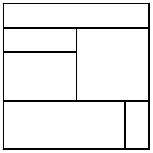}}\,\,.\]
Figure~\ref{coprod fig} shows each path $\gamma$ and the associated $A_\gamma$ and $B_\gamma$.
The figure uses the abbreviation ``c.''\ to mean ``the sum over all completions of appropriate size.''  

The example in Figure~\ref{coprod fig} illustrates an important property of the coproduct $\Delta_\dR$, which we now prove in general.
\begin{prop}\label{rec coeff 1}
Let $R$ be a diagonal rectangulation.
Then each term occurring with nonzero coefficient in $\Delta_\dR(R)$ occurs with coefficient 1.
\end{prop}
\begin{proof}
By definition, for each fixed path $\gamma$ from the top-left corner of $R$ to the bottom-right corner of $R$, every nonzero coefficient in the sums $A_\gamma$ and $B_\gamma$ is 1.
Suppose $\gamma_1$ and $\gamma_2$ are distinct paths from the top-left corner of $R$ to the bottom-right corner of $R$.
The term $\emptyset\otimes R$ occurs in $\Delta_\dR(R)$ with coefficient 1, so we need not consider the case where $\gamma_1$ or $\gamma_2$ passes through the bottom-left corner of $S$.
Let $R_1$ be a decomposition of $S$ obtained as a completion of the segments below $\gamma_1$ and let $R_2$ be a decomposition of $S$ obtained as a completion of the segments below~$\gamma_2$.
To prove the proposition, it is enough to show that $R_1$ and $R_2$ are distinct diagonal rectangulations.

Starting from the top-left corner of $R$, consider the first point $x$ at which $\gamma_1$ and $\gamma_2$ diverge.
Without loss of generality, $\gamma_1$ continues downward from $x$ and $\gamma_2$ continues to the right from $x$.
At some point $y$ directly below $x$, the path $\gamma_1$ turns to the right.
There are four possible local configurations at $y$:
The point $y$ can be in the interior of the right edge of a rectangle of $R$, in the interior of the left edge of $S$, in the interior of the top edge of a rectangle of $R$, or in the interior of the bottom edge of $S$.  

Suppose $y$ is in the interior of the right edge of a rectangle $U$ of $R$.
The decomposition $R_1$ has a rectangle $U_1$ containing $U$, and the decomposition $R_2$ has a rectangle $U_2$ containing $U$.
Since $R_1$ and $R_2$ are identical below and to the left of $y$, if the decompositions $R_1$ and $R_2$ are the same diagonal rectangulation, then the rectangles $U_1$ and $U_2$ correspond.
For each horizontal wall $W$ of $R$ having its left endpoint contained in the right edge of $U$, strictly lower than the point $y$, there is a horizontal wall of $R_1$ having the same left endpoint (contained in the right edge of $U_1$), and furthermore, these are the only horizontal walls of $R_1$ having left endpoints contained in the right edge of $U_1$.
For each of these horizontal walls $W$ of $R$ with left endpoint in the right edge of $U$, strictly lower than $y$, there is also a horizontal wall of $R_2$ with the same left endpoint (contained in the right edge of $U_2$).
But there is at least one more horizontal wall whose left endpoint is contained in the right edge of $U_2$, namely a horizontal wall of $R_2$ whose left endpoint is $y$.
We conclude that $R_1$ and $R_2$ are not the same diagonal rectangulation.

If $y$ is in the interior of the left edge of $S$, then there is no rectangle of $R$ to the left of $y$, but by considering walls whose left endpoint is in the left edge of $S$, we reach the same conclusion.

Now suppose $y$ is in the interior of the top edge of a rectangle $U$ of $R$, let $U_1$ be the rectangle of $R_1$ containing $U$, and let $U_2$ be the rectangle of $R_2$ containing $U$.
Every vertical wall of $R_1$ whose bottom endpoint is contained in the top edge of $U_1$ corresponds to a vertical wall of $R$, lying to the left of $y$, whose bottom endpoint is contained in the top edge of $U$.
For each of these vertical walls of $R$ there is a vertical wall of $R_2$ whose bottom endpoint is contained in the top edge of $U_2$, but there is at least one additional vertical wall of $R_2$ with bottom endpoint in the top edge of $U_2$:  the wall whose bottom endpoint is $y$.
Thus $R_1$ and $R_2$ are not the same diagonal rectangulation.

If $y$ is in the interior of the bottom edge of $S$, then there is no rectangle of $R$ below $y$, but arguing similarly, we reach the same conclusion.
\end{proof}

\section{The isomorphism from tBax to dRec}\label{isom sec}
In this section, we give a bijection between twisted Baxter permutations and diagonal rectangulations and show that the product and coproduct on diagonal rectangulations, defined in Section~\ref{Rec operations sec}, coincide, via the bijection, with the product and coproduct on tBax.
In particular, dRec is a Hopf algebra.
We then consider some natural automorphisms and antiautomorphisms of the Hopf algebra dRec.

\subsection{Bijection}\label{bij sec}
The bijection from twisted Baxter permutations to diagonal rectangulations is the restriction of a surjection $\rho$ from $S_n$ to $\dRec_n$.
Given $x\in S_n$, the diagonal rectangulation $\rho(x)$ is obtained by ``concatenating'' the bottom planar binary tree $\rho_b(x)$ with the top planar binary tree $\rho_t(x)$ as we now make precise:
Construct both trees with their leaves evenly spaced at intervals of $\sqrt{2}$, with all lines drawn straight and all angles right.
Rotate both trees clockwise by an eighth-turn, place the root of $\rho_b(x)$ at the origin and place the root of $\rho_t(x)$ at the point $(n,n)$.
Inspection of the definitions of $\rho_b$ and $\rho_t$ leads immediately to the conclusion that the union of the edges of the two rotated and translated trees defines a decomposition of the square into $n$ rectangles, each of whose interiors intersects the diagonal.
By Proposition~\ref{alt char}, this is a diagonal rectangulation of size $n$.
Inspection of these definitions also leads immediately to the following direct description of the map $\rho$.

Let $x\in S_n$.
Start with $n-1$ diagonal points in the square $S$ and number the spaces between them $1,2,\ldots,n$ from top-left to bottom-right.
For convenience in the description of $\rho$, we also include the top-left and bottom-right corners of $S$ as ``diagonal points.''
We read $x_1\cdots x_n$ from left to right and draw a rectangle for each entry.
After $i-1$ steps in the construction, let $T_{i-1}$ be the union of the left and bottom edges of $S$ with the union of the rectangles drawn in the first $i-1$ steps.
Once the construction is described, it is apparent by induction that each $T_{i-1}$ is left- and bottom-justified.
We draw an additional rectangle whose top-left and bottom-right corners are described as follows:

Consider the label $x_i$ on the diagonal.
If the diagonal point $p$ immediately above/left of the label $x_i$ is not in $T_{i-1}$, then the top-left corner of the new rectangle is the rightmost point of $T_{i-1}$ that is directly left of $p$.
If $p$ is in $T_{i-1}$ (necessarily on the boundary of $T_{i-1}$, then the top-left corner of the new rectangle is the highest point of $T_{i-1}$ directly above $p$.
Similarly, if the diagonal point $p'$ immediately below/right of the label $x_i$ is not in $T_{i-1}$, then the bottom-right corner of the new rectangle is the highest point of $T_{i-1}$ that is directly below $p$.
If $p'$ is in $T_{i-1}$ then the bottom-right corner of the new rectangle is the rightmost point of $T_{i-1}$ that is directly to the right of $p'$.
Figure~\ref{rho fig} illustrates the steps in the construction of $\rho(x)$ when $x=467198352$.  
In each step, the new rectangle is shown in red (the darkest gray when not viewed in color), and the set $T_{i-1}$ is shaded in medium gray.
The part of $S$ not yet covered by rectangles is shaded in light gray.

\begin{figure}[t]
\begin{tabular}{ccccccc}
\scalebox{.95}{\includegraphics{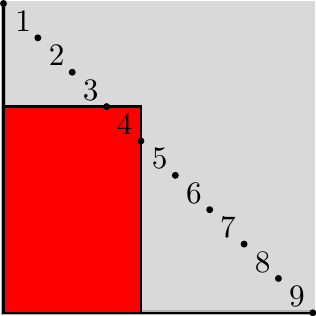}}&&&
\scalebox{.95}{\includegraphics{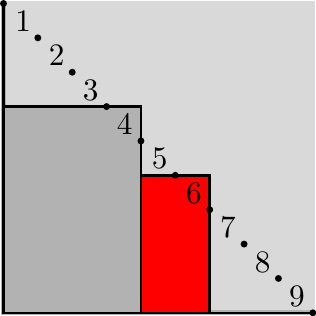}}&&&
\scalebox{.95}{\includegraphics{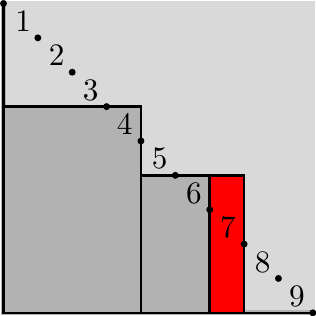}}\\[10 pt]
\scalebox{.95}{\includegraphics{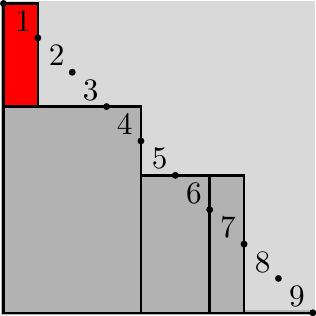}}&&&
\scalebox{.95}{\includegraphics{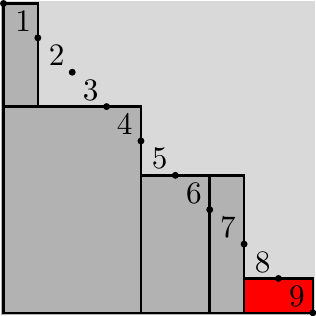}}&&&
\scalebox{.95}{\includegraphics{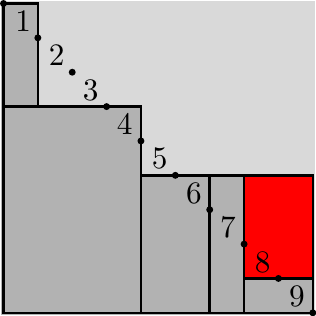}}\\[10 pt]
\scalebox{.95}{\includegraphics{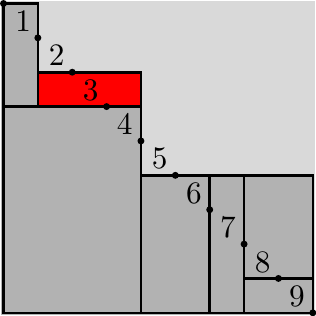}}&&&
\scalebox{.95}{\includegraphics{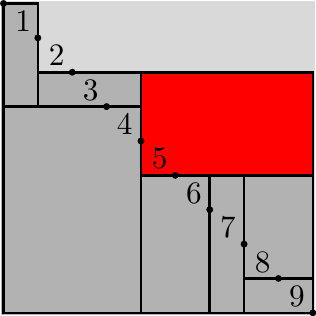}}&&&
\scalebox{.95}{\includegraphics{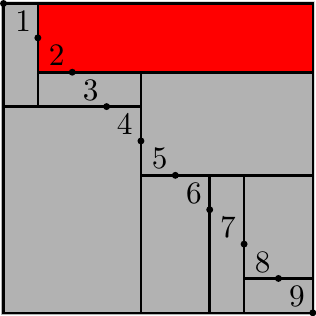}}
\end{tabular}
\caption{Steps in the construction of $\rho(467198352)$}
\label{rho fig}
\end{figure}

Our next goal is to prove the following theorem:

\begin{theorem}\label{rho bij}
The map $\rho$ restricts to a bijection between twisted Baxter permutations in $S_n$ and diagonal rectangulations of size $n$.
\end{theorem}

Theorem~\ref{rho bij} follows from two propositions.

\begin{prop}\label{rho surj}
The map $\rho:S_n\to\dRec_n$ is surjective.
\end{prop}
\begin{proof}
Let $R$ be a diagonal rectangulation.
We construct a permutation $x\in S_n$ with $\rho(x)=R$.
By Proposition~\ref{alt char}, $R$ is a decomposition of the square into $n$ rectangles, each of whose interiors intersects the diagonal.
Label the rectangles of $R$ as $1,2,\ldots,n$ from top-left to bottom-right along the diagonal.

For $i\in[n]$, assume that the entries $x_1x_2\cdots x_{i-1}$ have been determined, and that steps 1 through $i-1$ in the description of $\rho$, applied to $x_1x_2\cdots x_{i-1}$ give exactly the rectangles of $R$ labeled by $\set{x_1x_2\cdots x_{i-1}}$.
We try to choose a rectangle whose label can be taken as $x_i$.
Let $T$ be the set $T_{i-1}$ in the description of the $i\th$ step of $\rho$.
That is, $T$ is the union of the left and bottom edges of $S$ with the rectangles of $R$ labeled by $\set{x_1x_2\cdots x_{i-1}}$.

The analogous set $T_i$ must in particular be left- and bottom-justified.
Thus the rectangle we choose in step $i$ must have both its entire left edge and its entire bottom edge contained in $T$.
We claim that this condition on the rectangle is also sufficient.
In other words, we claim that, as long as we take $x_i$ to label a rectangle $U$ whose left and bottom edges are in $T$, the $i\th$ step of $\rho$ reproduces $U$.
To see this, we need only verify that the top-left and bottom-right corners of $U$ are as described in the definition of $\rho$.

The boundary of $U$ intersects the diagonal in exactly two points.  
Let $p$ be the top-left of the two, and let $p'$ be the bottom-right of the two.
If $p\in T$ then necessarily $p$ is in the left edge of $U$.
If the top corner of $U$ is somewhere below the highest point of $T$ directly above $p$, then the top-left corner of $U$ is the bottom-left corner of some rectangle $U'$ of $R$, but then the interior of $U'$ does not intersect the diagonal, which would be a contradiction.
If $p\not\in T$, then since the left edge of $U$ is contained in $T$, the point $p$ is on the top edge of $U$.
Since the left edge of $U$ is in $T$, the top-left corner of $U$ is in $T$, and is necessarily the rightmost point of $T$ that is directly left of $p$.
In either case, we see that the top-left corner of $U$ is the point described in the direct description of $\rho$.
The symmetric argument proves the analogous fact about the bottom-right corner of $U$, and the claim is proved.

In light of the claim, it remains only to argue that there exists some rectangle $U$ of $R$ which is not contained in $T$ but whose bottom and left edges are contained in $T$.
The boundary of $T$ includes a polygonal path from the top-left corner of $S$ to the bottom-right corner of $S$, always moving directly right or directly down.
At each corner where the path turns from moving down to moving right, there is a rectangle of $R$ which is not contained in $T$.
Number these rectangles $U_1,\ldots,U_k$ from top-left to bottom-right.
We show that there is a $j\in[k]$ such that the left and bottom edges of $U_j$ are contained in $T$.
The following three simple observations are illustrated in Figure~\ref{edge fig}.  
\begin{figure}[t]
\psfrag{1}[cc][cc]{\large$U_1$}
\psfrag{j}[cc][cc]{\large$U_j$}
\psfrag{j1}[lc][lc]{\large$U_{j+1}$}
\psfrag{k}[cc][cc]{\large$U_k$}
\psfrag{T}[cc][cc]{\large$T$}
\centerline{\scalebox{.8}{\includegraphics{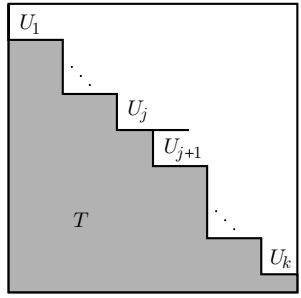}}}
\caption{An illustration for the proof of Proposition~\ref{rho surj}}
\label{edge fig}
\end{figure}
First, the left edge of $U_1$ is contained in $T$.
Second, for every $j\in[k-1]$, if the bottom edge of $U_j$ is not contained in $T$, then the left edge of $U_{j+1}$ is contained in $T$.
Third, the bottom edge of $U_k$ is contained in $T$.
Thus, considering the rectangles in order $U_1,\ldots,U_k$, we eventually find a rectangle $U_j$ whose left and bottom edges are contained in $T$.
\end{proof}

\begin{prop}\label{rho classes}
The fibers of the map $\rho:S_n\to\dRec_n$ are the congruence classes of $\Theta_\tB$.
\end{prop}
\begin{proof}
Proposition~\ref{rho b cong} established that the fibers of $\rho_b$ are the congruence classes of a congruence $\Theta_{231}$.
Similarly, by Proposition~\ref{rho t cong}, the fibers of $\rho_t$ form a congruence $\Theta_{312}$.
Two permutations $x,y\in S_n$ have $\rho(x)=\rho(y)$ if and only if both $\rho_b(x)=\rho_b(y)$ and $\rho_t(x)=\rho_t(y)$.
Thus the fibers of $\rho$ constitute the meet, in the partition lattice, of $\Theta_{231}$ and $\Theta_{312}$.
Proposition~\ref{10.2} identifies this meet as $\Theta_\tB$.
\end{proof}

Now, each $\Theta_\tB$-class has a unique minimal representative, and the set of these minimal representatives is the set of twisted Baxter permutations by Proposition~\ref{tBax props}.\ref{tBax bottoms}.
This completes the proof of Theorem~\ref{rho bij}.

The proof of Theorem~\ref{rho bij} leads easily to an explicit description of the inverse of the restriction of $\rho$.
Let $R$ be a diagonal rectangulation.
The proof of Proposition~\ref{rho surj} constructs a permutation $x$, one entry at a time, such that $\rho(x)=R$.
At each step, there were choices.
Let $\tau(R)$ be the permutation obtained by always choosing the smallest possible entry at each step.
In other words, choose the top/leftmost allowable rectangle from the list $U_1,\ldots,U_k$.

\begin{prop}\label{rho inv}
For any diagonal rectangulation $R$, the permutation $\tau(R)$ is a twisted Baxter permutation.
Thus, the map $\tau$ is the inverse of the restriction of $\rho$ to twisted Baxter permutations.
\end{prop}
It is a simple and instructive exercise to suppose that the permutation $\tau(R)$ contains the pattern $3$-$41$-$2$ or the pattern $2$-$41$-$3$ and to obtain a contradiction (\emph{i.e.}\ to show that, at some step, we failed to choose the top/leftmost rectangle).
However, this proposition is also an easy consequence of Propositions~\ref{tBax props}.\ref{tBax bottoms} and~\ref{rho classes}.
\begin{proof}
Let $x=x_1\cdots x_n$ be any permutation covered by $\tau(R)=t_1\cdots t_n$ in the weak order, and let $i,i+1$ be the indices where $x$ and $\tau(R)$ differ.
In particular, $x_1\cdots x_{i-1}=t_1\cdots t_{i-1}$ and $x_i<t_i$.
But $t_i$ is, by definition, the minimal entry of any permutation starting with $t_1\cdots t_{i-1}$ and mapping to $R$ under $\rho$.
Thus $\rho(x)\neq R$.
By Proposition~\ref{rho classes}, the fibers of $\rho$ are the $\Theta_\tB$-class, and are in particular intervals.
We conclude that $\tau(R)$ is minimal in its $\Theta_\tB$-class so that $\tau(R)$ is a twisted Baxter permutation by Proposition~\ref{tBax props}.\ref{tBax bottoms}.
\end{proof}

\begin{remark}\label{alt desc}
Reflection in the diagonal permutes the set of diagonal rectangulations.
The map $\rho$ seems not to respect this symmetry, because it builds a diagonal rectangulation starting from the bottom-left corner.
However, we might just as well have built diagonal rectangulations from the top-right corner.
Indeed, if $\rho'$ is the map taking $x\in S_n$ to the diagonal reflection of $\rho(x)$, one easily verifies that $\rho'=\rho\circ\rp$, where $\rp$ is the ``reverse-positions'' map defined in Section~\ref{auto sec}.
Thus the original map $\rho$ has an alternate description, reading $x$ from right to left, and drawing a top-right-justified rectangle for each entry.
\end{remark}

\begin{remark}\label{credit remark}
As mentioned in Section~\ref{lr sec}, various versions of the maps from permutations to trees have appeared in various sources.
The idea of mapping a permutation to a pair of ``twin'' binary trees to get pairs counted by the Baxter number is due to \cite{DGStack}, where the pairs are also counted by a non-intersecting lattice path argument.
The idea of putting the two trees together to form a rectangulation appears to be due to \cite{ABP}.
A map essentially equivalent to $\rho$ is used in~\cite{FFNO} in a bijection between pairs of twin binary trees and Baxter permutations, via diagonal rectangulations.
More is said about the map from~\cite{FFNO} in Remark~\ref{FFNO rem}.
A map related to $\rho$ also appears in \cite[Section~3]{ABP2}. 
\end{remark}

\begin{remark}\label{adjacency poset}
Each diagonal rectangulation $R$ of size $n$ corresponds to a certain partial order on the set $[n]$.
Specifically, first define a partial order on the rectangles of $R$ as follows:
Put $U_1<U_2$ if the bottom or left edge of $U_2$ intersects the top or right edge of $U_1$.
The proof of Proposition~\ref{rho surj} serves to verify that this relation is acyclic, so we obtain a partial order on the rectangles of $R$ by taking the reflexive and transitive closure.
Since the rectangles of $R$ are labeled by $[n]$, this is a partial order on $[n]$ which we call the \emph{adjacency poset} of $R$.
The proof of Proposition~\ref{rho surj} also shows that the fiber $\rho^{-1}(R)$ is the set of all linear extensions of the adjacency poset of $R$.
It would be interesting to give an intrinsic characterization of the posets that arise as adjacency posets of diagonal rectangulations.
In particular, the Hopf algebra dRec (or the isomorphic Hopf algebra tBax) could then be given an alternate combinatorial realization on the graded vector space with graded basis indexed by adjacency posets.
\end{remark}

\subsection{Isomorphism}\label{isom subsec}
In this section, we prove the following theorem:
\begin{theorem}\label{tBax dRec isom}
The triple $\dRec=(\K[\dRec_\infty],\bullet_\dR,\Delta_\dR)$ is a graded Hopf algebra.
The map $\rho$ is an isomorphism of Hopf algebras from tBax to dRec. 
\end{theorem}
Comparison of the examples in Sections~\ref{tBax Hopf subsec} and~\ref{Rec operations sec} provides an illustration of the theorem.
We need the following lemma.

\begin{lemma}\label{restrict rec}
Let $x\in S_p$ and $y\in S_q$ and let $z$ be a term in $x\bullet_S y$.
Then $\rho(x)=\TL_p(\rho(z))$ and $\rho(y)=\BR_q(\rho(z))$.
\end{lemma}
\begin{proof}
Let $n=p+q$.
It is immediate from the description of $\rho$ that the rectangles of $\TL_p(\rho(z))$ are the intersections of $[0,p]\times[q,n]$ with rectangles constructed from entries $z_j=i$ of $z$ with $i\in\set{1,\ldots,p}$.
Furthermore, $\TL_p(\rho(z))$ depends only on~$x$, (\emph{i.e.}\ only on the relative order of the values $1$ through $p$ in $z$).
Similarly, $\BR_q(\rho(z))$ depends only on $y$.
Since $\TL_p(\rho(z))=\rho(x)$ and $\BR_q(\rho(z))=\rho(y)$ when $z_1\cdots z_n=x_1\cdots x_p(y_1+p)\cdots(y_q+p)$, the lemma follows.
\end{proof}

\begin{proof}[Proof of Theorem~\ref{tBax dRec isom}] 
It is enough to show that $\rho$ is an algebra homomorphism and a coalgebra homomorphism.
Then Theorem~\ref{rho bij} implies that $\rho$ is an isomorphism, and thus dRec is a Hopf algebra because tBax is.

We first verify that $\rho(x\bullet_\tB y)=\rho(x)\bullet_\dR\rho(y)$. 
Let $x\in \tBax_p$ and $y\in \tBax_q$.
Since both $\bullet_\tB$ and $\bullet_\dR$ are sums with all coefficients one, the equation $\rho(x\bullet_\tB y)=\rho(x)\bullet_\dR\rho(y)$ holds if and only if the set of terms appearing in $\rho(x\bullet_\tB y)$ equals the set of terms appearing in $\rho(x)\bullet_\dR\rho(y)$.
Lemma~\ref{restrict rec} says that every term $\rho(z)$ in $\rho(x\bullet_\tB y)$ is also a term in $\rho(x)\bullet_\dR\rho(y)$.

On the other hand, let $R$ be a term in $\rho(x)\bullet_\dR\rho(y)$, so that $\rho(x)=\TL_p(R)$ and $\rho(y)=\BR_q(R)$.
Let $z=\tau(R)$, so that $\rho(z)=R$.
The restriction $x'$ of $z$ to the values $1,\ldots,p$ is a twisted Baxter permutation by Proposition~\ref{tBax restrict}, and Lemma~\ref{restrict rec} says that $\rho(x')=\TL_p(R)$.
Since $\rho$ is a bijection from $\tBax_p$ to $\dRec_p$, we conclude that $x=x'$, so that $x$ is the restriction of $z$ to the values $1,\ldots,p$.
Similarly, $y$ is the standardization of the restriction of $z$ to the values $p+1,\ldots,n$.
Thus $z$ is a term in $x\bullet_\tB y$, and so $R$ is a term in $\rho(x\bullet_\tB y)$. 

We now show that $(\rho\otimes\rho)(\Delta_\tB(z))=\Delta_\dR(\rho(z))$. 
Let $z\in\tBax_n$ and let $R=\rho(z)$.
For $p+q=n$, let $x\in\tBax_p$ and let $y\in\tBax_q$.
In light of Propositions~\ref{perm coeff 1} and~\ref{rec coeff 1}, we only need to show that the term $x\otimes y$ occurs in $\Delta_\tB(z)$ if and only if the term $\rho(x)\otimes \rho(y)$ occurs in $\Delta_\dR(R)$.

First, suppose that $\rho(x)\otimes \rho(y)$ is a term in $\Delta_\dR(R)$, and let $\gamma$ be the path such that $\rho(x)$ is a term in $A_\gamma$ and $\rho(y)$ is a term in $B_\gamma$.
We show that $x\otimes y$ is a term in $\Delta_\tB(z)$ by exhibiting a permutation $w$ such that $x=\std(w_1\cdots w_p)$, $y=\std(w_{p+1}\cdots w_q)$, and $\rho(w)=R$.
Numbering the rectangles of $R$, as usual, from $1$ to $n$ according to their intersections with the diagonal, from top-left to bottom-right, let $\Gamma$ be the set of numbers whose rectangles appear below/left of~$\gamma$.
Then there exists a unique permutation $x'$ of $\Gamma$ such that $\std(x')=x$ and a unique permutation $y'$ of $[n]\setminus\Gamma$ such that $\std(y')=y$.
We define $w$ to be the permutation in $S_n$ such that $w_1\cdots w_p=x'$ and $w_{p+1}\cdots w_n=y'$.

The permutation $x$ describes a sequence of rectangles in $\rho(x)$ such that each rectangle in the sequence has its left and bottom edges contained in the union of the preceding rectangles (or in the left or bottom edge of $S$).
But the rectangles below/left of $\gamma$ in $R$ are exactly the rectangles obtained as the intersection of a rectangle of $\rho(x)$ with the region below/left of $\gamma$.
Thus the sequence $x'$ describes a sequence of rectangles in $R$ such that each rectangle in the sequence has its left and bottom edges contained in the union of the preceding rectangles (or in the left or bottom edge of $S$).
Arguing as in the proof of Proposition~\ref{rho surj}, the first $p$ steps in the construction of $\rho(w)$ agree with $R$.
In other words, $\rho(w)$ agrees with $R$ below/left of $\gamma$.
We can apply the same argument to the alternate description of $\rho$ discussed in Remark~\ref{alt desc} to conclude that the first $q$ steps of the alternate description of $\rho(w)$ agree with $R$, so that $\rho(w)$ agrees with $R$ above $\gamma$.
Thus $\rho(w)=R$, or in other words, $w$ is in the $\Theta_\tB$-class of $z$.
We conclude that $x\otimes y$ is a term in $\Delta_\tB(R)$.

Conversely, suppose $x\otimes y$ is a term in $\Delta_\tB(z)$.
Thus there exists $w$ in the $\Theta_\tB$-class of $z$ such that $x=\std(w_1\cdots w_p)$ and $y=\std(w_{p+1}\cdots w_n)$.
Let $\gamma$ be the path in $R$ defined as the upper-right boundary of the union of the rectangles constructed in the first $p$ steps of the construction of $\rho(w)$, together with the left and bottom edges of $S$.
As discussed in connection with the definition of $\Delta_\dR$, if $p>0$, then there are exactly $p-1$ walls of $R$ that extend below/left of~$\gamma$.
The $p-1$ walls meet the diagonal in a set $X'$ of $p-1$ distinct points.
We now show that $\rho(x)$, constructed on the square $S$ with diagonal points $X'$, agrees with $R$ below/left of~$\gamma$.

We argue by induction on $p$.
The base case $p=1$ is easy, so suppose $p>1$.
Let $\tilde{\gamma}$ be the path separating rectangles constructed in steps $1$ through $p-1$ in the construction of $\rho(w)$ from the other rectangles of $R$.
Let $\tilde{X}'$ be the subset of $X$ constructed from $\tilde{\gamma}$ just as $X'$ was constructed from~$\gamma$.
Let $\tilde{R}$ be the diagonal rectangulation $\rho(\std(w_1\cdots w_{p-1}))$, as constructed on $S$, with respect to the diagonal points $\tilde{X}'$.
By induction, $\tilde{R}$ agrees with $R$ in the region strictly below/left of $\tilde{\gamma}$.

Consider the alternate description (see Remark~\ref{alt desc}) of $\rho(x)=\rho(\std(w_1\cdots w_p))$, as constructed on the square $S$, with respect to the diagonal points $X'$ and the alternate description of $\tilde{R}=\rho(\std(w_1\cdots w_{p-1}))$, as constructed with respect to the diagonal points $\tilde{X}'$.
Each process constructs a diagonal rectangulation by drawing a sequence of top- and right-justified rectangles.
Comparison of the two processes reveals that $\tilde{R}$ agrees with $\rho(x)$ in the region $S\setminus U$, where $U$ is the first top- and right-justified rectangle constructed in the alternate description of $\rho(x)$. 
We conclude that $\rho(x)$ agrees with~$R$ in the region strictly below/left of $\tilde{\gamma}$.
The rectangle between $\gamma$ and $\tilde{\gamma}$ is the part of $U$ that is below $\gamma$, so $\rho(x)$ agrees with~$R$ below $\gamma$.

In light of the alternate description of $\rho$, the symmetric proof shows that $\rho(y)$ and $R$ agree above $\gamma$.
We have showed that $\rho(x)\otimes\rho(y)$ is a term in $\Delta_\dR(R)$.
This completes the proof of the theorem.
\end{proof}

\begin{remark}\label{Rec auto}
Remark~\ref{tBax auto} discusses automorphisms and antiautomorphisms of tBax.
These exhibit themselves as (anti)automorphisms of dRec given by symmetries of the underlying square.  
The algebra antiautomorphism/coalgebra automorphism $\rv$ of tBax corresponds to the algebra antiautomorphism/coalgebra automorphism of dRec obtained by reflecting each rectangulation along the line segment connecting the lower-left corner of $S$ to the upper-right corner.
The algebra automorphism/coalgebra antiautomorphism $\rp$ of tBax corresponds to reflecting each rectangulation along the diagonal.
(See Remark~\ref{alt desc}.)
Thus the Hopf-algebra antiautomorphism $\rv\circ\rp$ corresponds to rotating each rectangulation one half-turn.
\end{remark}

\begin{remark}
The Hopf algebra LR is embedded as a Hopf subalgebra of dRec.
Specifically, a bottom tree $T$ maps to the sum of all diagonal rectangulations whose bottom half is $T$.
The analogous map embeds LR$'$ as a Hopf subalgebra of dRec.
These embeddings can be used to give descriptions of the product and coproduct in LR and LR$'$ analogous to the definition of the product and coproduct in dRec.
It is interesting to compare these descriptions to the descriptions, from~\cite[Section~1.5]{AgSoLR}, of the (dual) product and coproduct in LR$^*$.
\end{remark}

\section{The lattice and polytope of diagonal rectangulations}\label{lattice rec sec}
The weak order on $S_n$ modulo the congruence $\Theta_\tB$ is a lattice on the set of $\Theta_\tB$-classes.
A general fact about finite lattices, mentioned previously in Section~\ref{pattern sec}, implies that this quotient is isomorphic to the subposet of the weak order induced by the set twisted Baxter permutations.
Since $\Theta_\tB$-classes are in bijection with diagonal rectangulations, this quotient also defines a lattice structure on diagonal rectangulations.
We use the symbol $\dRec_n$ to stand for the set $\dRec_n$ endowed with this quotient lattice structure.
In this section, we describe the lattice of diagonal rectangulations, as a partial order, in terms of local moves on rectangles, analogous to the \emph{diagonal flips} on triangulations that define the Tamari lattice.
These local moves were already considered in \cite[Section~2]{ABP}.
We also construct a polytope, analogous to the associahedron, whose vertices are indexed by diagonal rectangulations.

As an example of the lattice of diagonal rectangulations, consider the case $n=4$.
The weak order on $S_4$ is illustrated in Figure~\ref{S4 weak fig}.
Applying $\rho$ to each element of $S_4$, we obtain all 22 elements of $\dRec_4$.
The permutations $3412$ and $3142$ both map to the same diagonal rectangulation and the permutations $2413$ and $2143$ both map to the same diagonal rectangulation.
The partial order $\dRec_4$ is shown in Figure~\ref{Rec4 fig}.

\begin{figure}[t]
\centerline{\includegraphics{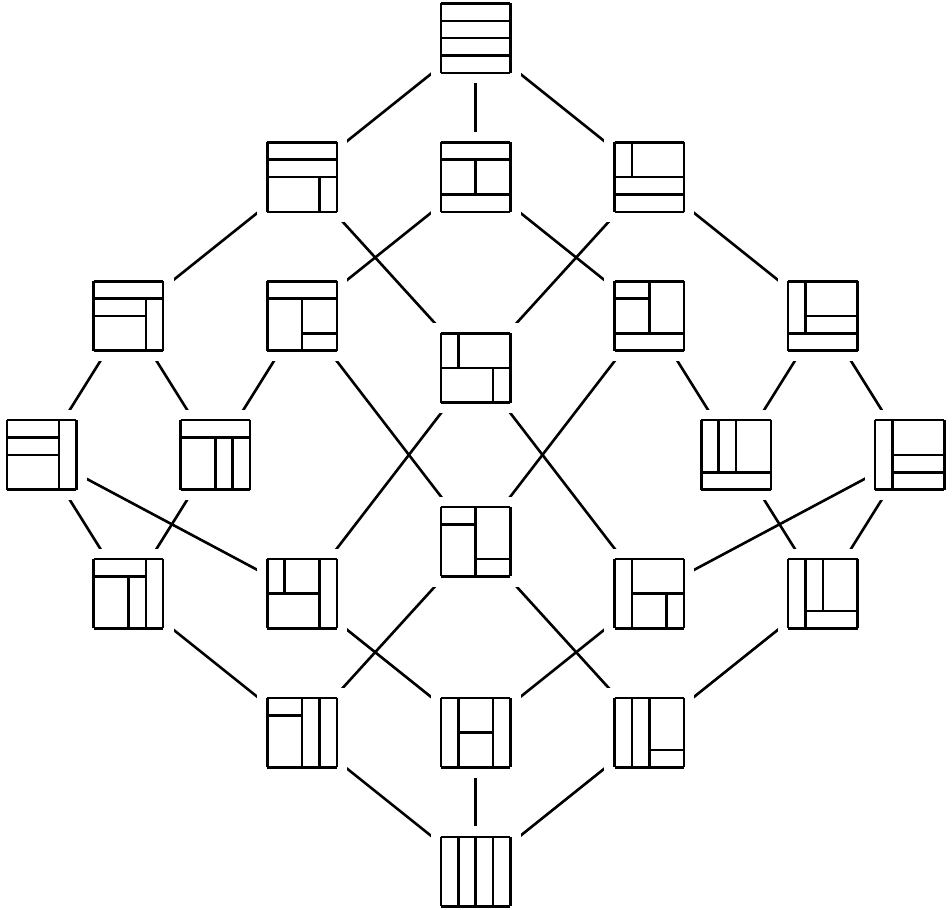}}
\caption{The lattice $\dRec_4$ of diagonal rectangulations of size 4}
\label{Rec4 fig}
\end{figure}

To describe the partial order $\dRec_n$ directly in terms of diagonal rectangulations, we need to define some further terminology.
\begin{figure}[t]
\centerline{\scalebox{.92}{\includegraphics{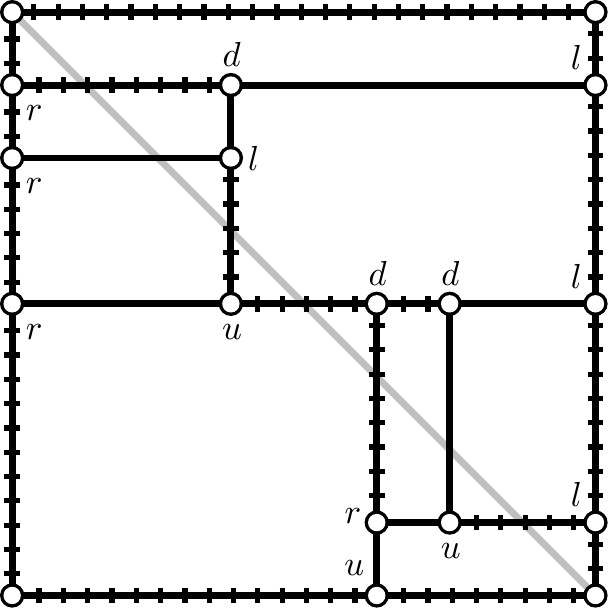}}}
\caption{An illustration of terminology related to pivots}
\label{pivot fig}
\end{figure}
Let $R$ be a diagonal rectangulation of size $n$.
A point is a \emph{vertex of $R$} if it is the vertex of some rectangle in $R$.
Vertices are shown as white dots in Figure~\ref{pivot fig}.
A line segment $E$ is an \emph{edge of $R$} if it is contained in the edge of some rectangle of $R$, if its endpoints are vertices of $R$ and if it contains no vertex of $R$ in its interior.

Every vertex $v$ of $R$, with the exception of the corners of $S$, is incident to exactly three edges, and we say that $v$ is a(n) \emph{up, down, right or left vertex} according to the direction of the unmatched edge (the edge which is not opposite another edge incident to $v$).
The vertices in Figure~\ref{pivot fig} are labeled ``u,'' ``d,'' ``r,'' or ``l'' for up, down, right or left.
Every left or down vertex must be above the diagonal and every right or up vertex must be below the diagonal.
(Otherwise, some rectangle incident to the vertex has its interior disjoint from the diagonal, in contradiction to Proposition~\ref{alt char}.)
If a vertex $v$ is not a corner of $S$, then exactly two edges incident to $v$ exit $v$ in the direction of the diagonal.  
One is the unmatched edge; we say that the other edge leaving $v$ towards the diagonal is \emph{locked} by $v$.
In addition, we declare that any edge incident to a corner of $S$ is locked by that corner.
It is easily verified that, with this convention, every edge contained in the boundary of $S$ is locked by some vertex.
The locked edges in Figure~\ref{pivot fig} are overwritten with hash marks.

An edge of $R$ can participate in a \emph{pivot} if and only if it is not locked.
The pivots are described as follows.
Let $E$ be a non-locked edge of $R$.

If $E$ is an unmatched edge of both of its endpoints, then the endpoints of $E$ are necessarily either a left vertex and a right vertex or a down vertex and an up vertex.
Furthermore, $E$ must, in this case, cross the diagonal, necessarily at a point $p$ in the set $X$.
There are two rectangles $U_1$ and $U_2$ of $R$ such that $E$ is an edge of $U_1$ and an edge of $U_2$.
To pivot the edge $E$, remove it and insert the unique other edge $E'$ containing $p$ that cuts the rectangle $U_1\cup U_2$ into two rectangles.
The result is another diagonal rectangulation $R'$ of size $n$.
The endpoints of $E$ are not vertices of $R'$, but they are replaced by two new vertices, the endpoints of $E'$.
This type of pivot is called a \emph{diagonal pivot}, because $E$ and $E'$ share a point on the diagonal.

Suppose $E$ is a matched edge of one of its endpoints $v$.
Then since $E$ is not locked, $E$ is the matched edge of $v$ which leaves $v$ in the direction away from the diagonal, and $E$ is the unmatched edge of its other endpoint $v'$.
In particular, the two rectangles $U_1$ and $U_2$ of $R$ whose boundaries contain $E$ together form a hexagon with five convex right angles and one concave right angle.
To pivot $E$, remove it and insert the unique other edge $E'$ that cuts $U_1\cup U_2$ into two rectangles.
Again, the pivot produces a diagonal rectangulation $R'$ of size $n$.
The point $v$ is a vertex of $R'$ but $v'$ is not.
It is replaced by a new vertex, the other endpoint of $E'$.
This type of pivot is called a \emph{vertex pivot}, because $E$ and $E'$ share a vertex of $R$ and of~$R'$.

In both cases, pivoting the edge $E$ produces a new edge $E'$ in $R'$, such that $E'$ is not locked and such that pivoting the edge $E'$ transforms $R'$ back to $R$.
Also, in both cases, a vertical edge is replaced by a horizontal edge or vice-versa.
There are eight rectangulations that are obtained from the diagonal rectangulation of Figure~\ref{pivot fig} by pivoting.
These are shown in Figures~\ref{pivot down} and~\ref{pivot up}.  
In each picture, the new edge appears in red (or gray).
We now state the main result of this section.

\begin{figure}
\begin{tabular}{ccccc}
\includegraphics{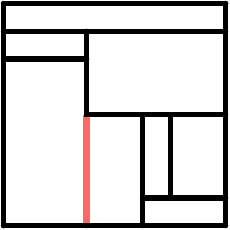}&&\includegraphics{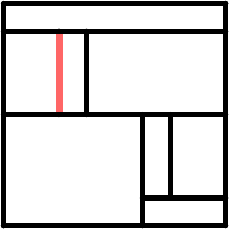}&&\includegraphics{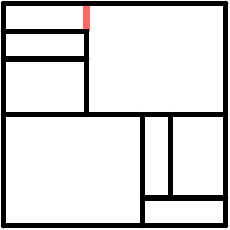}
\end{tabular}

\vspace{10 pt}

\begin{tabular}{ccc}
\includegraphics{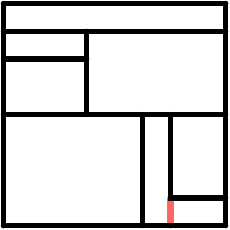}&&\includegraphics{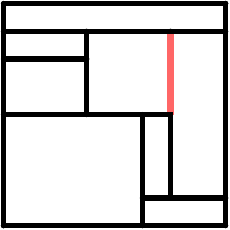}
\end{tabular}

\caption{Diagonal rectangulations covered by the diagonal rectangulation of Figure~\ref{pivot fig}}
\label{pivot down}
\end{figure}

\begin{figure}
\begin{tabular}{ccccc}
\includegraphics{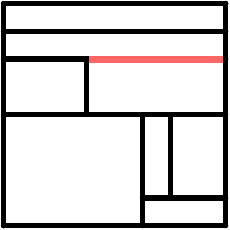}&&\includegraphics{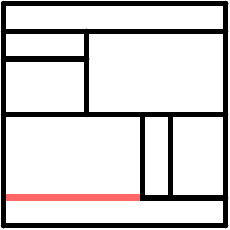}&&\includegraphics{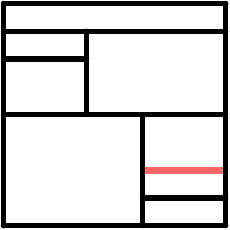}
\end{tabular}
\caption{Diagonal rectangulations which cover the diagonal rectangulation of Figure~\ref{pivot fig}}
\label{pivot up}
\end{figure}

\begin{theorem}\label{pivots}
Two diagonal rectangulations $R$ and $R'$ of size $n$ have $R\covered R'$ in $\dRec_n$ if and only if they are related by a pivot such that the pivoted edge in $R$ is vertical.
\end{theorem}
Theorem~\ref{pivots} can be proved by appealing to Proposition~\ref{10.2} and to descriptions of covers in the weak order modulo $\Theta_{231}$ and in the weak order modulo $\Theta_{312}$ that can be obtained as specializations of \cite[Proposition~5.3]{cambrian}.
However, to avoid having to translate these descriptions from the language of triangulations to the dual language of planar binary trees, we argue directly in terms of diagonal rectangulations.

\begin{proof}
Let $L$ be a finite lattice, let $\Theta$ be a congruence on $L$, and let $y$ be the minimal element in its $\Theta$-class.
Then \cite[Proposition~2.2]{con_app} states that a $\Theta$-class $C$ is covered by the class of $y$ in $L/\Theta$ if and only if some element of $C$ is covered by $y$ in $L$.
Specializing this fact to the weak order and the congruence $\Theta_\tB$, we rephrase the theorem as the following assertion:
For any $y\in\tBax_n$, $R'=\rho(y)$ and $R$ are related by pivots as in the statement of the theorem if and only if $R=\rho(x)$ for some permutation $x$ covered by $y$ in the weak order.

Let $y\in\tBax_n$ and suppose $x\covered y$ in the weak order.
Thus $x$ and $y$ differ only in positions $i$ and $i+1$, and $x_i=y_{i+1}<y_i=x_{i+1}$.
Consider the partial diagonal rectangulation created from either $x$ or $y$ in the first $i-1$ steps of $\rho$ and how the next two steps proceed for $x$ and $y$.
As in the proof of Proposition~\ref{rho surj}, we let $T$ be the union of the left and bottom edges of $S$ with the rectangles produced in the first $i-1$ steps.
We let $X_i$ and $X_{i+1}$ be the rectangles produced in steps $i$ and $i+1$ of the construction of $\rho(x)$ and define $Y_i$ and $Y_{i+1}$ analogously for $\rho(y)$.
Since $y$ is a twisted Baxter permutation, one of the following three cases applies.
In every case, it is apparent that, after step $i+1$, the constructions of $\rho(x)$ and $\rho(y)$ proceed identically.

\noindent\textbf{Case 1}: $y_i=y_{i+1}+1$.
The rectangles $Y_i$ and $Y_{i+1}$ have a diagonal point $p$ in common, and $p$ is in the interior of an edge of each rectangle.
Let $L$ be the (necessarily horizontal) edge of $R'$ containing $p$.
(Thus $L$ is the intersection of the edge of $Y_i$ containing $p$ with the edge of $Y_{i+1}$ containing $p$.)
If the endpoint~$q$ of $L$ that is above/right of the diagonal is a down vertex, then the rectangle whose top-left corner is~$q$ must be chosen after $Y_i$ but before $Y_{i+1}$.
If the endpoint $r$ of $L$ that is below/left of the diagonal is an up vertex, then the rectangle whose bottom-right corner is $r$ must be chosen after $Y_i$ but before $Y_{i+1}$.
Since $Y_{i+1}$ is chosen immediately after $Y_i$ in the construction of $\rho(y)$, we conclude that~$q$ is a left vertex, that $r$ is a right vertex, and therefore that $Y_i\cup Y_{i+1}$ is a rectangle.
The same argument shows that $X_i\cup X_{i+1}$ is a rectangle, and it is easy to see that the rectangles $X_i\cup X_{i+1}$ and $Y_i\cup Y_{i+1}$ coincide.
Thus $\rho(x)$ and $\rho(y)$ are related by a diagonal pivot.

\noindent\textbf{Case 2}: $y_i>y_{i+1}+1$ and every $a$ with $y_{i+1}<a<y_i$ occurs to the left of position $i$ in $y_1\cdots y_n$.
In this case, all rectangles labeled $a$ with $y_{i+1}<a<y_i$ are in $T$.
Thus, following the boundary of $T$ from the bottom-left corner of $Y_{i+1}$ to the bottom-left corner of $Y_i$ we move to the right to a point that we call $p$, and then we move down, making no other turns.  
The point $p$ is the top-left corner of $Y_i$, and the top-right corner of $Y_i$ is the bottom-right corner of $Y_{i+1}$.
On the other hand, $p$ is the bottom-right corner of $X_i$, and the top-right corner of $X_i$ is the top-left corner of $X_{i+1}$.
Thus $\rho(x)$ and $\rho(y)$ are related by a vertex pivot, where the pivot-vertex is $p$.

\noindent\textbf{Case 3}: $y_i>y_{i+1}+1$ and every $a$ with $y_{i+1}<a<y_i$ occurs to the right of position $i+1$ in $y_1\cdots y_n$.
In this case, no rectangles labeled $a$ with $y_{i+1}<a<y_i$ are in $T$, and the boundary of $T$ moves first down, then right as we traverse it from the top-left corner of $Y_{i+1}$ to the bottom-right corner of $Y_i$.
Similarly to Case 2, we see that $\rho(x)$ and $\rho(y)$ are related by a vertex pivot.

In every case, the pivoted edge is vertical in $\rho(x)$ and horizontal in $\rho(y)$.

Conversely, suppose that $y\in\tBax_n$ and that $R'=\rho(y)$ and $R$ are related by pivots as in the statement of the theorem.
Since the pivoted edge in $R$ is vertical, the pivoted edge in $R'$ is horizontal, so one of the rectangles involved in the pivot is lower than the other in $R'$.
Call the lower rectangle $Y_i$, and let $y_i$ be the entry of $y$ associated to $Y_i$ by $\rho$.
The other rectangle, $Y_j$ is above $Y_i$ and so is associated to a later entry $y_j$ of $y$ with $y_j<y_i$.
The rectangles $Y_i$ and $Y_j$ can stand in three possible relationships, illustrated in Figure~\ref{config}.
\begin{figure}
\psfrag{Yi}{$Y_i$}
\psfrag{Yj}{$Y_j$}
\begin{tabular}{ccccc}
\scalebox{.98}{\includegraphics{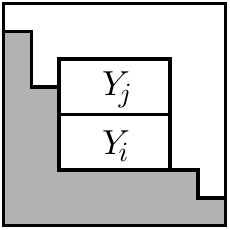}}&&\scalebox{.98}{\includegraphics{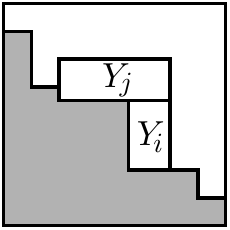}}&&\scalebox{.98}{\includegraphics{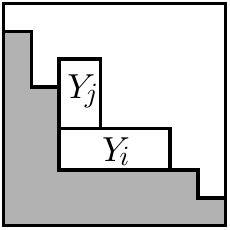}}
\end{tabular}
\caption{An illustration for the proof of Theorem~\ref{pivots}}
\label{config}
\end{figure}
Again let $T$ be the union of the left and bottom edges of $S$ with the rectangles produced in the first $i-1$ steps of constructing $\rho(y)$.
The region $T$ is indicated in gray in Figure~\ref{config}.
\emph{A priori}, the left edge of $Y_j$ may not be contained in $T$, but we argue below that it is.

Since $y=\tau(R')$, in the sense of Proposition~\ref{rho inv}, $y_{i+1}$ is the smallest value in $\set{y_{i+1},\ldots,y_n}$ such that the rectangle $Y_{i+1}$ labeled $y_{i+1}$ has its left and bottom edges contained in the union $T\cup Y_i$ of the left and bottom edges of $S$ with the rectangles chosen in steps 1 to $i$.
We claim that $j=i+1$.
The claim is equivalent to two assertions:
First, the assertion that the left and bottom edges of $Y_j$ are contained in $T\cup Y_i$; and second, the assertion that $y_j$ is the smallest value in $\set{y_{i+1},\ldots,y_n}$ such that the rectangle $Y_j$ has this property.
If either assertion fails, then inspection of each case illustrated in Figure~\ref{config} reveals that some other rectangle, above/left of $Y_i$, could have been chosen in step $i$, contradicting the fact that $y=\tau(R')$.

In light of the claim, we can take $x$ to be the permutation obtained from $y$ by swapping $y_i$ and $y_{i+1}$, so that $x\covered y$ in the weak order.
Inspection of the three cases illustrated in Figure~\ref{config} (corrected to show that the left edge of $Y_j$ is contained in $T$) shows that $R=\rho(x)$.
\end{proof}

We now describe the construction of a polytope whose vertices are indexed by diagonal rectangulations and whose edges correspond to pivots.
The construction involves fans, particularly normal fans of polytopes; we refer the reader to \cite[Lecture~7]{Ziegler} for background.
Consider the arrangement $\Arr$ of hyperplanes in $\reals^n$ given by equations $x_i=x_j$ for each $i$ and $j$ with $1\le i<j\le n$.
This is called the \emph{braid arrangement} or the \emph{Coxeter arrangement for $S_n$}.
The complement of $\cup_{H\in\Arr}H$ in $\reals^n$ consists of $n!$ connected components, whose closures are called \emph{regions}, and which correspond naturally to the elements of $S_n$:
For each permutation $x_1x_2\cdots x_n\in S_n$, there is a unique region of $\Arr$ whose interior contains the vector $(x_1,x_2,\ldots,x_n)\in\reals^n$, and each region of $\Arr$ contains exactly one such vector.
The regions are the maximal cones of a complete fan $\F$, which is the normal fan of a polytope called the \emph{permutohedron}.

Any lattice congruence $\Theta$ on the weak order can be interpreted as an equivalence relation on regions.
Consider the union of the regions in some equivalence class.
As a special case of \cite[Theorem~5.1]{con_app}, each such union is a convex cone, and furthermore these cones are the maximal cones of a fan $\F_\Theta$ which coarsens $\F$.
Hohlweg and Lange \cite[Proposition~4.1]{HohLan} showed that the fans $\F_{\Theta_{231}}$ and $\F_{\Theta_{312}}$ are the normal fans of polytopes which we call $P_{231}$ and $P_{312}$ respectively.
Both $P_{231}$ and $P_{312}$ are realizations of the \emph{associahedron}, meaning that their vertices are indexed by the set of triangulations of an $(n+2)$-gon and their edges correspond to diagonal flips.
Proposition~\ref{10.2} says that $\F_{\Theta_{\tB}}$ is the coarsest fan that refines both $\F_{\Theta_{231}}$ and $\F_{\Theta_{312}}$.
Thus a well-known fact on Minkowski sums (for example, see \cite[Proposition~7.12]{Ziegler}) implies that $\F_{\Theta_{\tB}}$ is the normal fan of the polytope, which we call $P_{\dRec}$, given by the Minkowski sum $P_{231}+P_{312}$.
Now \cite[Proposition~3.3]{con_app} and \cite[Theorem~5.1]{con_app} combine to say that the Hasse diagram of $\dRec_n$ orients the $1$-skeleton of $P_{\dRec}$.
Thus by Theorem~\ref{pivots}, the vertices of $P_{\dRec}$ are indexed by diagonal rectangulations and the edges of $P_{\dRec}$ correspond to pivots.

\begin{remark}\label{connected}
Since $\dRec_n$ is a finite lattice, its Hasse diagram is in particular connected.
In light of Theorem~\ref{pivots}, we recover the observation, originally made in \cite[Lemma~2.3]{ABP}, that the edge-pivoting graph on diagonal rectangulations of size $n$ is connected.
Since this graph is also the $1$-skeleton of the $(n-1)$-dimensional polytope $P_{\dRec}$, Balinski's Theorem (for example, see \cite[Theorem~3.14]{Ziegler}) implies the following stronger result:
The graph on diagonal rectangulations of size $n$ is $(n-1)$-connected, meaning that the removal of at most $n-2$ vertices leaves a connected graph.
\end{remark}

\begin{remark}\label{2n+2}
Edge-pivoting preserves the number of vertices of a diagonal rectangulation.
We thus easily conclude that every diagonal rectangulation of size $n$ has exactly $2n+2$ vertices.
\end{remark}

\begin{remark}\label{properties}
Beyond the results mentioned above, \cite[Theorem~5.1]{con_app} implies additional detailed combinatorial information about $\dRec_n$.
In particular, open intervals in $\dRec_n$ are either homotopy equivalent to spheres or are contractible, and thus M\"{o}bius function values are in $\set{-1,0,1}$.
Furthermore, $\dRec_n$ is the partial order on the maximal cones of $\F_{\Theta_{\tB}}$ induced by a certain linear functional.
For definitions and details, see~\cite[Sections~2--5]{con_app}.
\end{remark}

\section{Baxter permutations and twisted Baxter permutations}\label{Bax tBax sec}
In Section~\ref{tBax sec}, we saw that the definitions of twisted Baxter permutations and Baxter permutations are closely related syntactically.
Specifically, recall that a permutation is a twisted Baxter permutation if and only if it avoids both $3$-$41$-$2$ and $2$-$41$-$3$, and that a permutation is a Baxter permutation if and only if it avoids both $3$-$14$-$2$ and $2$-$41$-$3$.
While syntactically similar definitions often produce quite different mathematical objects, we have already seen that the two definitions produce sets of permutations sharing the same cardinality (the Baxter number).
In this section, we explain the relationship between Baxter permutations and twisted Baxter permutations by giving a bijection that is built directly upon the syntactic relationship.
As has been the case throughout the paper, lattice congruences of the weak order play a key role.
The bijection from Baxter permutations to twisted Baxter permutations allows us to carry both a natural Hopf algebra structure and a natural lattice structure over to Baxter permutations.

\subsection{Bijection}\label{Bax tBax bij}
To define the bijection between twisted Baxter permutations and Baxter permutations, we appeal to yet another special case of the results of \cite[Section~9]{con_app}.
There is a lattice congruence $\Theta_{3412}$ on each $S_n$ with associated projection maps $\pidown^{3412}$ and $\piup_{3412}$ such that the family of congruences $\Theta_{3412}$ is an $\H$-family, and thus has the following properties:

\begin{prop}\label{3412 props}

\noindent
\begin{enumerate}
\item \label{3412 bottom}
A permutation is the minimal element in its $\Theta_{3412}$-class if and only if it avoids $3$-$41$-$2$.
\item \label{3412 top}
A permutation is the maximal element in its $\Theta_{3412}$-class if and only if it avoids $3$-$14$-$2$.
\item Suppose $x\covered y$ in the weak order.  Then $x\equiv y$ modulo $\Theta_{3412}$ if and only if $x$ is obtained from $y$ by a $(3412\to3142)$-move.
\item Let $y\in S_n$.
If some permutation $x\in S_n$ can be obtained from $y$ by a $(3412\to3142)$-move, then $\pidown^{3412}(y)=\pidown^{3412}(x)$.
Otherwise, $\pidown^{3412}(y)=y$.  
\item Let $x\in S_n$.
If some permutation $y\in S_n$ can be obtained from $x$ by a $(3142\to3412)$-move, then $\piup_{3412}(x)=\piup_{3412}(y)$.
Otherwise, $\piup_{3412}(x)=x$. 
\end{enumerate}
\end{prop}
The main result of this section is the following theorem.

\begin{theorem}\label{Bax tBax}
The restriction of $\pidown^{3412}$ is a bijection from Baxter permutations to twisted Baxter permutations.
Its inverse is the restriction of $\piup_{3412}$.
\end{theorem}

To prove the theorem, we prove two lemmas:
\begin{lemma}\label{2413 lemma}
Suppose $x$ is the bottom element of some $\Theta_{3412}$-class and $y$ is the top element of the same $\Theta_{3412}$-class, so that $\pidown^{3412}(y)=x$ and $\piup_{3412}(x)=y$.
If $y$ avoids $2$-$41$-$3$ then $x$ avoids $2$-$41$-$3$.
\end{lemma}

\begin{lemma}\label{Rec to Bax}
Every $\Theta_\tB$-class contains a Baxter permutation.
\end{lemma}

Given Lemmas~\ref{2413 lemma} and~\ref{Rec to Bax}, Theorem~\ref{Bax tBax} follows as we now explain.
\begin{proof}[Proof of Theorem~\ref{Bax tBax}]
Assertions~\ref{3412 bottom} and~\ref{3412 top} of Proposition~\ref{3412 props} say in particular that a Baxter permutation is the top element of its $\Theta_{3412}$-class and a twisted Baxter permutation is the bottom element of its $\Theta_{3412}$-class.
Now Lemma~\ref{2413 lemma} implies that $\pidown^{3412}$ maps Baxter permutations to twisted Baxter permutations.
Thus, since $\pidown^{3412}$ is a bijection from top elements of $\Theta_{3412}$-classes to bottom elements of $\Theta_{3412}$-classes, the restriction of $\pidown^{3412}$ to Baxter permutations is a one-to-one map to the set of twisted Baxter permutations.
If $x$ is a twisted Baxter permutation, then Lemma~\ref{Rec to Bax} says that there is a Baxter permutation $y$ in the $\Theta_\tB$-class of $x$.
Thus, $\pidown^{3412}(y)$ is a twisted Baxter permutation.
Comparing Propositions~\ref{tBax props} and~\ref{3412 props}, we see that $\Theta_\tB$-classes are unions of $\Theta_{3412}$-classes, so $\pidown^{3412}(y)$ is in the $\Theta_\tB$-class of $x$.
Since each twisted Baxter permutation is the unique minimal element of its $\Theta_\tB$-class, we have $x=\pidown^{3412}(y)$.
Furthermore, $y=\piup_{3412}(x)$ because $y$ is the top element of its $\Theta^{3412}$-class.
\end{proof}
We now prove Lemmas~\ref{2413 lemma} and~\ref{Rec to Bax}.

\begin{proof}[Proof of Lemma~\ref{2413 lemma}]
We prove the lemma by proving the following assertion:  
If $y$ is not minimal in its $\Theta_{3412}$-class and $y$ avoids $2$-$41$-$3$, then there is a strictly lower element $y'$ in the $\Theta_{3412}$-class of $y$ which also avoids $2$-$41$-$3$.

Indeed, a $y$ that is not minimal in its $\Theta_{3412}$-class contains an instance of the pattern $3$-$41$-$2$.
That is, there exists a subsequence $cdab$ of $y_1\cdots y_n$ with $a<b<c<d$ and with $d$ immediately preceding $a$ in $y_1\cdots y_n$.
Let $e$ be the rightmost element left of $d$ which is less than $d$.
Possibly $e$ coincides with $c$, and otherwise $e$ is to the right of $c$.
Let $f_1\cdots f_i$ be the sequence of elements between $e$ and $d$.
Possibly $i=0$, in which case the sequence $f_1\cdots f_i$ is empty.
Let $h$ be the leftmost element right of $a$ which is greater than $a$.
Let $g_1\cdots g_j$ be the sequence of elements between $a$ and $h$.
Again, possibly $h=b$ and/or $j=0$.
Define a permutation $y'$ by removing the sequence $a \,g_1\cdots g_j$ from $y_1\cdots y_n$ and replacing it between $e$ and $f_1$.
Thus we write
\[y=y_1\cdots \,c\,\cdots \underbrace{\raisebox{0 pt}[8 pt][5 pt]{$e$}}_{<d}\, \underbrace{\raisebox{0 pt}[8 pt][5 pt]{$f_1\cdots f_i$}}_{>d}\, d\,a \,\underbrace{\raisebox{0 pt}[8 pt][5 pt]{$g_1\cdots g_j$}}_{<a}\, \underbrace{\raisebox{0 pt}[8 pt][5 pt]{$h$}}_{>a}\cdots \,b\, \cdots y_n\]
and
\[y'=y_1\cdots \,c\,\cdots \underbrace{\raisebox{0 pt}[8 pt][5 pt]{$e$}}_{<d}\,a \,\underbrace{\raisebox{0 pt}[8 pt][5 pt]{$g_1\cdots g_j$}}_{<a}\, \underbrace{\raisebox{0 pt}[8 pt][5 pt]{$f_1\cdots f_i$}}_{>d}\, d\, \underbrace{\raisebox{0 pt}[8 pt][5 pt]{$h$}}_{>a}\cdots \,b\, \cdots y_n\]
Since every entry in the sequence $a \,g_1\cdots g_j$ is $\le a$ and every element of the sequence $f_1\cdots f_i\, d$ is $\ge d$, we can obtain $y'$ from $y$ by a sequence of $(3412\to3142)$-moves.
In each move, $c$ plays the role of ``$3$,'' an entry from the sequence $f_1\cdots f_i\, d$ plays the role of ``$4$,'' an entry from $a \,g_1\cdots g_j$ plays the role of ``$1$,'' and $b$ plays the role of ``$2$.''
Thus by Proposition~\ref{3412 props}, $y'$ is in the $\Theta_{3412}$-class of $y$.

We now show that if $y'$ contains a $2$-$41$-$3$-pattern, then $y$ also contains a \mbox{$2$-$41$-$3$}-pattern.
Suppose that $y'$ contains a subsequence $\beta\delta\alpha\gamma$ that is a $2$-$41$-$3$-pattern.
That is, $\alpha<\beta<\gamma<\delta$ and $\delta$ immediately precedes $\alpha$ in $y'$.
We break into six easy cases, based on the location of $\alpha$ in $y'$.

\noindent
\textbf{Case 1:}  $\alpha$ is $e$ or some entry to the left of $e$.
Then $\beta\delta\alpha\gamma$ is a $2$-$41$-$3$-pattern in $y$ as well.

\noindent
\textbf{Case 2:}  $\alpha$ coincides with $a$.
Then $\delta$ coincides with $e$ and $\gamma<\delta=e<d$.
Since each $f$ is $>d$, no $f$ coincides with $\gamma$, and furthermore $\gamma$ is not $d$.
Thus $\beta da\gamma$ is a $2$-$41$-$3$-pattern in $y$.

\noindent
\textbf{Case 3:} $\alpha$ coincides with some $g$.
Then $\delta$ either $a$ or some $g$, and therefore $\delta\le a$.
In particular, since each $f$ is $>d$, and since $\gamma<\delta\le a<d$, no $f$ coincides with $\gamma$.
Therefore $\beta\delta\alpha\gamma$ is a $2$-$41$-$3$-pattern in $y$ as well.

\noindent
\textbf{Case 4:} $\alpha$ coincides with some $f$ or with $d$.
Note that $\alpha$ does not coincide with $f_1$ (and if $j=0$, then $\alpha$ is not $d$) because $\delta$ immediately precedes $\alpha$ and $\delta>\alpha$.
Now $\alpha\ge d>a$, but each $g$ is $<a$, so $\beta$ is not some $g$ and is not $a$.
Therefore $\beta\delta\alpha\gamma$ is a $2$-$41$-$3$-pattern in $y$ as well.

\noindent
\textbf{Case 5:} $\alpha$ is $h$.
Then $\delta$ is $d$.
Also, $\alpha>a$, so $\beta>\alpha>a$, and thus $\beta$ does not coincide with any $g$ nor with $a$.
Thus $\beta da \gamma$ is a $2$-$41$-$3$-pattern in $y$.

\noindent \textbf{Case 6:} $\alpha$ is some entry right of $h$.
Then $\beta\delta\alpha\gamma$ is a $2$-$41$-$3$-pattern in $y$ as well.
\end{proof}

\begin{proof}[Proof of Lemma~\ref{Rec to Bax}] 
By Proposition~\ref{rho classes}, the $\Theta_\tB$-classes are the fibers of $\rho$.
Thus it is enough to show that, for every diagonal rectangulation $R$, there is a Baxter permutation~$x$ with $\rho(x)=R$.
As in the proof of Proposition~\ref{rho surj}, we construct the permutation one entry at a time.
In other words, we totally order the rectangles of $R$ so that each rectangle has its left and bottoms sides contained in the union of the previous rectangles with the left and bottom sides of $S$.

To order the rectangles, we use an algorithm first given in~\cite{FFNO}.  
(See Remark~\ref{FFNO rem} below.)
Let $U$ be the rectangle chosen in step $i-1$, let $p$ be the top-right corner of $U$, and let $T_{i-1}$ be the union of the rectangles chosen in steps $1$ through $i-1$ with the left and bottom sides of $S$.
Follow the boundary of $T_{i-1}$ left from $p$ until it turns upward, and let $U_{\A}$ be the rectangle whose bottom-left corner is the point where the boundary turns upward.
If the top edge of $U$ is contained in the boundary of~$S$, then $U_{\A}$ is not defined.
Similarly, follow the boundary of $T_{i-1}$ down from $p$ until it turns towards the right side of $S$, and let $U_{\B}$ be the rectangle whose bottom-left corner is the point where the boundary turns.
If the right edge of $U$ is contained in the boundary of $S$, then $U_{\B}$ is not defined.
If both the top and right edges of $U$ are in the boundary of $S$, then $i-1=n$ and we are done.
Otherwise, the top-right corner $p$ of $U$ is either a left vertex or a down vertex.
If $p$ is a left vertex, then choose $U_{\A}$ as the next rectangle in the total order, and if $p$ is a down vertex, then choose $U_{\B}$ next.
(This rule is consistent with the cases where only $U_{\A}$ or only $U_{\B}$ is defined.)

To complete the proof, we need to establish two assertions:
First, we need to show that it is possible to follow this algorithm to obtain a permutation $x$ with $\rho(x)=R$.
Specifically, we need to show that, at every step $i$, the algorithm succeeds in selecting a rectangle whose left and bottom sides are contained in $T_{i-1}$.
Second, we need to show that the permutation $x$ is a Baxter permutation.

Suppose that we have successfully followed the algorithm in steps $1$ to $i-1$, choosing rectangle $U$ is step $i-1$.
Suppose that the top-right corner $p$ of $U$ is a left vertex, so that $U_{\A}$ is chosen in step $i$.
Since $p$ is a left vertex, there is a rectangle $V$ whose bottom-right corner is $p$.
(It is possible that $V=U_{\A}$.)
Let $L$ be the vertical line containing $p$.
The key point in the argument is the following immediate observation:
For every $j\ge i$, if we have successfully continued the algorithm in steps $i$ through $j-1$ without ever choosing rectangle $V$, then the rectangle chosen in step $j$ has its interior completely to the left of $L$.

We now show that, if we have successfully followed the algorithm in steps $i-1$, then the algorithm also succeeds, in step $i$, in picking a rectangle with its left and bottom sides contained in $T_{i-1}$.
By the symmetry of reflection perpendicular to the diagonal, we may as well assume that $p$ is a left vertex, so that the chosen rectangle is $U_{\A}$.
Since $p$ is a left vertex, the bottom edge of $U_{\A}$ is contained in $T_{i-1}$.
Suppose that the left edge of $U_{\A}$ is not contained in $T_{i-1}$.
Let $p'$ be the highest point in the intersection of $U_{\A}$ with $T_{i-1}$, let $U'$ be the rectangle whose top-right corner is $p'$ and let $V'$ be the rectangle whose bottom-right corner is $p'$.
In particular, $V'$ has not been chosen before step $i$, because $p'$ is the highest point in the intersection of $U_{\A}$ with $T_{i-1}$.
Since $U$ was chosen just before $U_{\A}$, it was in particular chosen after~$U'$, but then since $p'$ is a left vertex, the observation of the previous paragraph says in particular that $U$ could not have been chosen before $V'$.
This contradiction shows that the left edge of $U_{\A}$ is contained in $T$.

We have shown that the algorithm constructs a permutation $x$ with $\rho(x)=R$.
It remains to show that the permutation $x$ is a Baxter permutation.
Suppose $x$ contains a $3$-$14$-$2$-pattern: a subsequence $x_hx_{i-1}x_ix_j$ of $x_1\cdots x_n$ such that $x_{i-1}<x_j<x_h<x_i$.
As before, let $U$ be the rectangle added in step $i-1$ and let $p$, $T_{i-1}$, and $U_{\B}$ be as above.
Since the next rectangle chosen is labeled $x_i$, this rectangle is $U_{\B}$.
In particular, moving downwards from $p$ along the boundary of $T_{i-1}$. we pass to the right of the rectangle $U'$ chosen in step $h$.
But the rectangle $U''$ chosen later, in step $j$, is between $U$ and $U'$ on the diagonal, and since $T_{i-1}$ is left- and bottom-justified we see that the interiors of $U''$ and $T_{i-1}$ intersect.
This contradiction implies that $x$ contains no $3$-$14$-$2$-pattern.
If $x$ contains a $2$-$41$-$3$-pattern, 
then we reach a similar contradiction by an argument that is symmetric with respect to reflection perpendicular to the diagonal.
\end{proof}

\begin{remark}\label{FFNO rem}
The algorithm given in the proof of Lemma~\ref{Rec to Bax} is essentially borrowed from the proof of \cite[Theorem~5.1]{FFNO}, which gives a bijection between Baxter permutations and twin binary trees.
The ideas used above to prove that the algorithm works and produces a Baxter permutation are also essentially the same.
However, the algorithm in \cite[Theorem~5.1]{FFNO} outputs the inverse of the permutation constructed in the proof of Lemma~\ref{Rec to Bax}.
For this reason, for completeness, and also because of the differences between Theorem~\ref{Bax tBax} and \cite[Theorem~5.1]{FFNO}, we reproduce the argument here.
\end{remark}

\subsection{Lattice and Hopf structures on Baxter permutations}\label{Bax lat Hopf sec}
In this section, use Theorem~\ref{Bax tBax} to study the subposet (in fact, lattice) of the weak order induced by Baxter permutations, and to put a Hopf algebra structure on Baxter permutations.

In a general finite lattice, it is well known, and easy to verify, that the downward and upward projection maps associated to a congruence are order-preserving.
Furthermore, the quotient modulo a congruence is isomorphic not only to the subposet induced by the bottom elements of congruence classes, but also, by symmetry, to the subposet induced by the top elements.
The downward projection map is an isomorphism from the restriction to top elements to the restriction to bottom elements, and the upward projection map is the inverse.
In the case of the congruence $\Theta_{3412}$, these considerations imply the following corollary of Theorem~\ref{Bax tBax}.

\begin{cor}\label{isom cor}
The restriction of $\pidown^{3412}$ is an isomorphism from the weak order restricted to Baxter permutations to the weak order restricted to twisted Baxter permutations.
The inverse isomorphism is $\piup_{3412}$.
In particular, the weak order on Baxter permutations is a lattice.
\end{cor}

Recall from Section~\ref{lattice rec sec} that the map $\rho$ is an isomorphism from the weak order restricted to twisted Baxter permutations to the lattice of diagonal rectangulations.
Thus $\rho\circ\pidown^{3412}$ is an isomorphism from the weak order restricted to Baxter permutations to the lattice of diagonal rectangulations.
However, the map $\rho$ is constant on $\Theta_\tB$-classes, and thus constant on $\Theta_{3412}$-classes, so that $\rho\circ\pidown^{3412}=\rho$.
Thus we have the following corollary, where $\beta(R)$ denotes the Baxter permutation produced from the diagonal rectangulation $R$ by the procedure described in the proof of Lemma~\ref{Rec to Bax}.

\begin{cor}\label{rec isom cor}
The restriction of $\rho$ is an isomorphism from the weak order restricted to Baxter permutations to the lattice of diagonal rectangulations, with inverse $\beta$.
\end{cor}

The bijection between twisted Baxter permutations and Baxter permutations also allows us to define a Hopf algebra structure on Baxter permutations.
Let $c^B$ be the map taking a Baxter permutation $y$ to the sum, in MR, of the permutations in the $\Theta_\tB$-class of $y$.
The elements of MR of the form $c^B(y)$ for Baxter permutations $y$ coincide with the elements $c^{\Theta_\tB}(x)$ for twisted Baxter permutations $x$.
As discussed in Section~\ref{tBax Hopf sec}, these elements form a graded basis for a graded Hopf subalgebra of MR.
Instead of realizing this Hopf subalgebra as a Hopf algebra structure on $\K[\tBax_\infty]$, we can just as well realize it as a Hopf algebra structure on $\K[\Bax_\infty]=\bigoplus_{n\ge 0}\K[\Bax_n]$.

The product on Baxter permutations is necessarily $x\bullet_B y=r^B(c^B(x)\bullet_Sc^B(y))$, where $r^B$ is the linear map on $\K[S_\infty]$ whose action on permutations is to fix Baxter permutations and to send all other permutations to zero.
This product is written more simply as $x\bullet_B y=r^B(x\bullet_Sy)$.
In other words, $x\bullet_By$ is obtained by summing over all shifted shuffles of $x$ and $y$, but ignoring those shuffles that are not Baxter permutations.
To see why this simpler formula for the product holds, notice that Proposition~\ref{Bax restrict} implies that a shuffle of two permutations $x$ and $y$ is not a Baxter permutation unless both $x$ and $y$ are Baxter permutations.
(A generalization of Proposition~\ref{tBax restrict} similarly accounts for the fact that the product described in Section~\ref{H fam sec} is $r^\Theta(x\bullet_Sy)$ rather than $r^\Theta(c^\Theta(x)\bullet_Sc^\Theta(y))$.)

The coproduct on Baxter permutations is $\Delta_B=(r^B\otimes r^B)\circ\Delta_S\circ c^B$.
Just as in tBax, this coproduct cannot be further simplified.
Thus the coproduct $\Delta_B(x)$ of a Baxter permutation $x$ is obtained by summing $\Delta_S(y)$ over all permutations $y$ in the $\Theta_\tB$-class of $x$ and then deleting all terms which contain permutations that are not Baxter.

\section*{Acknowledgments}\label{ack}
The authors thank Marcelo Aguiar for helpful conversations.

\end{document}